\documentclass[10pt,reqno]{amsart}

 \usepackage{amsmath,amsfonts,amssymb,amscd,amsthm,amsbsy,bbm, epsf,calc,enumerate,color,graphicx,verbatim,cite,import}
\usepackage{color}
\usepackage{datetime,appendix}
\usepackage{chapterbib,epstopdf}

\usepackage{float}

\usepackage{graphicx}

\usepackage{ifpdf}
\ifpdf
    \usepackage[colorlinks,citecolor=black,linkcolor=black]{hyperref}
\fi

\usepackage[mathlines]{lineno}

\newcommand{\lb}{\label}
\newcommand{\beq}{\begin{equation}}
\newcommand{\eeq}{\end{equation}}

\newtheorem{theorem}{Theorem}[section]

\newtheorem{proposition}[theorem]{Proposition}
\newtheorem{lemma}[theorem]{Lemma}
\newtheorem{corollary}[theorem]{Corollary}
\newtheorem{definition}[theorem]{Definition}

\theoremstyle{remark}
\newtheorem{remark}[theorem]{Remark}

\definecolor{darkgreen}{rgb}{0,0.4,0}

\def \l {\left(}
\def\r {\right)}

\numberwithin{equation}{section}

\def\XXint#1#2#3{{\setbox0=\hbox{$#1{#2#3}{\int}$ }
\vcenter{\hbox{$#2#3$ }}\kern-.6\wd0}}

\def \l {\left(}
\def\r {\right)}

\newcommand{\bbN}{{\mathbb{N}}}
\newcommand{\bbR}{{\mathbb{R}}}

\newcommand{\bbZ}{{\mathbb{Z}}}

\newcommand{\bbT}{{\mathbb{T}}}

\newcommand{\calL}{{\mathcal L}}

\newcommand{\calD}{{\mathcal D}}
\newcommand{\loc}{{\textup{loc}}}

\newcommand{\eps}{\varepsilon}

\title[Free boundary Regularity]{Free Boundary Regularity of the Porous Medium Equation with nonlocal drifts in Dimension One}
 
\author[Yuming Paul Zhang]{\bfseries  Yuming Paul Zhang}

\address{
(Y.P. Zhang) Department of Mathematics \\ 
University of California   \\ 
San Diego\\
USA}
\email{yzhangpaul@ucsd.edu}

\begin{document}

\vspace{18mm} \setcounter{page}{1} \thispagestyle{empty}

\begin{abstract}
We study the free boundary of the porous medium equation with nonlocal drifts in dimension one. Under the assumption that the initial data has super-quadratic
growth at the free boundary, we show that the solution is smooth in space and $C^{2,1}_{\loc}$ in time, and then the free boundary is $C^{2,1}_{\loc}$. Moreover if the drift is local, both the solution and the free boundary are smooth.
\end{abstract}

\maketitle

\vspace{.1cm}

\vspace{.1cm}

\bigskip

\section{Introduction}

In this paper, we study the free boundary problem of the following equation in space dimension one
\begin{equation}\label{main}
\left\{\begin{aligned}
&\varrho_t=\l \varrho^m\r_{xx}+(\varrho V+\varrho \, W*\varrho)_x &\text{ in }&\mathbb{R}\times \bbR^+, \\
& \varrho(x,0)=\varrho_{0}(x) &\text{ on }& \mathbb{R},
\end{aligned}
\right.
\end{equation}
where ${V},W$ are two smooth vector fields in $\mathbb{R}\times \mathbb{R}^+$ and $m>1$. The initial data $\varrho_0$ is non-negative, bounded and compactly supported. 

The nonlinear diffusion in \eqref{main} represents an anti-congestion effect. 
When $W\equiv 0$, the system models a gas flow in one direction in a 3-dimensional Euclidean space filled with homogeneous porous medium and $\varrho$ is the gas density, \cite{caffarelli1979regularity}. The vector field $-{V}$ corresponds to an external force (e.g. wind) acting on the flow \cite{BGHP,kimlei,carrillo2001entropy,chayes2013aggregation}.  
In mathematical biology, the nonlocal term involving $W$ appears in the model of chemotaxis, angiogenesis and motion of animal crowds, where the behavior of agents are largely driven by interaction forces, due to chemical or social effects \cite{TBL,carrillo2001entropy,carrillo2019aggregation,carrillo2019,blanchet2009critical}. The equation covers an important class of aggregation-diffusion equations with smooth kernels.



\medskip

The \textit{finite propagation} property is one well-known feature of the nonlinear diffusion, see section 14 \cite{book}. The property implies that if the non-negative solution $\varrho$ is initially compactly supported, then $\varrho$ stays compactly supported for all finite time. Therefore it makes sense to consider the solution's \textit{free boundary} (or \textit{interface}) that separates the region where there is gas ($\varrho>0$) from the void place ($\varrho=0$). 

Note that if the solution $\varrho$ is strictly positive, \eqref{main} is a uniformly parabolic equation and the nonlocal drift inherits the regularity of $\varrho$. Hence by iteration $\varrho$ is smooth. So problems occur near the free boundary. 
The goal of the paper is to study the regularity property of both the solution and the free boundary of \eqref{main}.

\medskip

Let us briefly discuss the regularizing mechanism of the interface. 
We start with rewriting the equation in the form of a continuity equation,
\begin{equation*}
\varrho_t+(\varrho \,(-u_x+{B}) )_x=0
\end{equation*}
where $u:=\frac{m}{m-1}\varrho^{m-1}$ is the pressure variable and 
\begin{equation}\label{B}
{B}:=-\l V+W*\varrho \r(x,t).
\end{equation}
Notice that the support of $\varrho$ is the same as the support of its pressure. So it suffices to consider $u$, and the discussions below are for general dimensions. 

In dimension $d$, the pressure variable $u$ satisfies
\begin{equation}\label{premain}
\left\{\begin{aligned}
&u_t=(m-1) u\,\Delta u+|\nabla u|^2-\nabla u\cdot{B}-(m-1)u \,\nabla\cdot{B} &\text{ in }&\mathbb{R}^d\times \bbR^+, \\
& u(x,0)=u_{0}(x) &\text{ on }& \mathbb{R}^d,
\end{aligned}
\right.
\end{equation}
where $u_0:=\frac{m}{m-1}\varrho_0^{m-1}$.
Formally, the speed of the free boundary to its outer normal direction (which equals $\frac{-\nabla u}{|\nabla u|}$) is the same as $\frac{u_t}{|\nabla u|}$. Let us denote the free boundary at time $t$ as $\Gamma_t(u):=\partial\{u(\cdot,t)>0\}$. If $u\Delta u=0$ at a free boundary point which can be induced from
\begin{equation}\label{cond: fund esti}
\Delta u(x,t)> -\infty \quad\text{ in }\quad\mathbb{R}^d\times\bbR^+,
\end{equation}
then \eqref{premain} implies that the velocity of the free boundary equals
\begin{equation}
    \label{1.7}
\frac{u_t}{|\nabla u|} 
=|\nabla u|- \frac{\nabla u\cdot B}{|\nabla u|} 
\end{equation}
since $u=0$ on $\Gamma_t(u)$.
Here
\eqref{cond: fund esti} is often referred to as the \textit{fundamental estimate} and \eqref{1.7} is called \textit{Darcy's law}.
Now if the free boundary is  \textit{non-degenerate}:
\begin{equation}
    \label{nondeg}
    \liminf_{y\to x}|\nabla u|(y,t)>0\quad\text{ for }\quad x\in\Gamma_t(u),
\end{equation}
in view of \eqref{1.7}, the free boundary expands with a positive speed relatively to $B$. This movement is expected to regularize the free boundary since it comes from the diffusion, and this is also strongly bonded to the regularity of the solution.
Both \eqref{cond: fund esti} and \eqref{nondeg} are crucial to the free boundary problem. 

\medskip

When ${V}=W\equiv 0$, the equation for all dimensions is the well-known {\it Porous Medium Equation} $(PME)$ and there is an extensive literature studying the regularity properties of the free boundary. 
The fundamental estimate of $(PME)$ is due to Aronson and Benilan \cite{fundamentalest}. Caffarelli and Friedman \cite{CFregularity} 
have shown that the solution's free boundary can be described by $t=S(x)$ where $S$ is H\"{o}lder continuous if $u_0$ satisfies \eqref{1.6} i.e. $u_0$ grows slightly faster than the quadratic-growth near the boundary. 
Later Caffarelli, V\'{a}zquez and Wolanski \cite{CVWlipschitz} prove that after a finite time, the free boundary is a Lipschitz continous $d$-dimensional surface, and furthermore if
\eqref{cond: fund esti} and \eqref{nondeg} hold at $t=0$,
the free boundary is non-degenerate for all time.
Based on non-degeneracy, Caffarelli and Wolanski proved that the interface is a $C^{1,\alpha}$ surface in \cite{C1alpha}.
Later Koch \cite{koch}, Aronson and V\'{a}zquez \cite{LV12} improved the regularity: the solution is actually smooth uniformly up to the free boundary and the free boundary is a smooth surface after the finite time. 
In dimension one, these results are known earlier in \cite{LV12,LV30,LV1}. In \cite{LV}, Lee and V\'{a}zquez found that in general dimensions the solution becomes concave in finite time. More recently
Kienzler, Koch and V\'{a}zquez \cite{kienzler2018flatness} proved that, without assuming non-degeneracy on the initial data, flatness of the solution implies smoothness of both the solution and the free boundary.


\medskip

Regarding the equation with drifts, well-posedness of \eqref{main} is established in \cite{bertozzi2009existence,bertsch,carrillo2001entropy}. If $V,W$ are potential vector fields, the equation shares the feature of being a gradient flow of a free energy functional, as discussed extensively in the literature ( \cite{ambrosio2008gradient,carrillo2019aggregation,zhang2018continuity,carrillo2019} etc.). We only mentioned a small portion of works in this active field of research and the topics range from regularity, asymptotics, singular limits, vanishing viscosity and phase transition etc. For the case when $V$ is smooth and $W\equiv 0$, it was proved in \cite{dib83,di1982continuity} that the solution becomes H\"{o}lder continuous instantaneously after time $0$. In view of \cite{hwang2019continuity,zhang2017regularity}, $V\in  {L^p_xL^q_t}$ locally  with $\frac{d}{p}+\frac{2}{q}>1$ is enough to deduce the continuity. 

To our knowledge, the regularity of the free boundary is widely open even for $W\equiv 0$. 
The presence of drifts poses significant challenges and it typically implies that one cannot rely on the classical results. 
Even for travelling wave solutions in space dimension $2$ with a smooth divergence-free drift, an interesting numerical experiment by Monsaigeon \cite{numerics} suggests the possibility of singular free boundary (while the travelling wave solutions of $(PME)$ are of the form $(x_1+ct)_+$ and the free boundary is just hyperplanes). The analysis of this observation is lacking, and we only know from \cite{traveling} the existence of Lipschitz traveling wave solution.
 When $-V$ is a convex potential vector field, Kim and Lei \cite{kimlei} showed the exponential convergence rate of the free boundary to the one of an equilibrium. The author and Kim \cite{kim2018fb} showed that in general dimensions with general smooth $V$ and $W\equiv 0$, the free boundary is locally non-degenerate and then $C^{1,\alpha}$ under a cone monotonicity condition on the solution and an upper bound on $u_t-C|\nabla u|$. 

This paper reports the first attempt to characterize the free boundary regularity in the presence of both local and nonlocal drifts. By making use of the huge advantage of space dimension one, we are able to obtain a Lipschitz regularity of the pressure variable (which is not known in general dimensions), and some regularity property of the nonlocal drifts. Next, since in dimension one the free boundary is just a collection of functions of time, we do not need the cone monotonicity condition assumed in \cite{kim2018fb} to prevent sudden topological changes in the free boundary. With the obtained regularity of the solution and the drift, we modify the method that is introduced in \cite{caffarelli1979regularity} to obtain the non-degeneracy under the assumption that the free boundary is strictly expanding relatively to the streamlines. 
Eventually,
we show that the solution is $C^{2,1}_{\loc}$ in time and smooth in space uniformly up to the free boundary, and the free boundary is $C^{2,1}_{\loc}$. 


\subsection{Our results}

The goal of this paper is to study the regularity property of both the solution and the free boundary of \eqref{main} in space dimension one.
We first show the following regularity (of the solution and the drift) that is uniform in $\bbR$ and locally uniform in time. 
Throughout the paper, we will assume that $V$ and $W$ are bounded in $C_{x,t}^k(\bbR\times [0,\infty))$ for each $k\geq0$.

\begin{theorem}\label{T.1.1}
{\rm [Lemmas~\ref{L.3.1}, \ref{L.3.2}]}
Let $\varrho$ be the solution to \eqref{main} with bounded, non-negative, compactly supported initial data $\varrho_0$, and with $V,W$ bounded in $ C^\infty_{x,t}(\bbR\times[0,\infty))$. Let $u$ be the pressure.
Then for any $T>\tau>0$,
there exists a constant $C=C(\tau,T)$ such that
\[\|u_x\|_{L^\infty(\mathbb{R}\times [\tau,T])}+\|u_t\|_{L^\infty(\mathbb{R}\times [\tau,T])}\leq C.\]
Let $B=B(x,t)$ be as given in \eqref{B}. For any $k\geq 0$, there exists $C'=C'(\tau,T,k)$ such that
\[\|B\|_{C^k_xC^{1,1}_t(\mathbb{R}\times [\tau,T])} \leq C'.\]

\end{theorem}

The obtained regularity of $B$ is essential for establishing non-degeneracy later.
The spatial regularity of $B$ follows immediately from the regularity of $V,W$ and that the solution is compactly supported. However the bounds on $B_t$ and $B_{tt}$ are more delicate and require more careful treatment because
\[
B_t=-\l V_t+W_t*\varrho+W* \varrho_t \r
\]
and it is possible that when $m>2$
\[
\varrho_t=\frac{1}{m}\l\frac{m-1}{m}u \r^\frac{2-m}{m-1} u_t
\]
is singular at $u=0$ even with smooth $u$.
Enlightened by Theorem 15.6 \cite{book}, we overcome this problem by proving that $\varrho_t(\cdot,t)\in L^p$ for some $p$ and for a.e. $t>0$. Using this identification of $\varrho_t(\cdot,t)$, we obtain that $B$ is Lipschitz continuous in time, and combining this with the equation \eqref{main}, the Lipschitz bound of $B_t$ follows. We will explain the obstacle that prevents us from further improving the regularity of $B$ after Theorem \ref{T.1.4}.

For solutions without compact support, to obtain the regularity, we only need to further assume some integrability condition on $W$ and derivatives of $W$, see Remark \ref{R.3.3}. The same applies to Theorems \ref{T.1.2}-\ref{T.1.4} below. For simplicity we mainly  discuss solutions with compactly supported initial data in the paper.

\medskip

Now we proceed to study the free boundary's regularity. We will firstly show that the free boundary is Lipchitz continuous and then justify Darcy's law in Lemma \ref{L.4.2}.

As suggested in \eqref{1.7}, the gas diffuses and at the same time flows along the drift. So it is important to consider the \textit{streamline} $X(x_0,t_0;t)$ that is defined to be the integral curve along the vector field $B$ starting at $(x_0,t_0)$ for a time period $t$ i.e.
$X(t):=X(x_0,t_0;t)$ is the unique solution to
\begin{equation}\label{4.1}
    \left\{\begin{aligned}
    &\partial_t X(t)=B(X(t),t_0+t)\\
    &X(0)=x_0.
    \end{aligned}\right.
\end{equation}
For $(PME)$ the positive zone of solutions is non-contracting, while when there is a drift, 
the positive zone is non-contracting relatively to streamlines. Actually we have the following stronger alternative result.
We use the notation
\[r(t):=\sup\{x\,|\, u(x,t)>0\}\]
as the right-hand side free boundary of $u$.


\begin{proposition}\label{P.1.2}
Let $x_0=r(t_0)$ for some $t_0>0$. Then either of the following holds:
\begin{itemize}
    \item[\normalfont{(i)} ] $r(t)=X(x_0,t_0;t-t_0)$ for all $t\in [0,t_0]$;
    \medskip
    \item[\normalfont{(ii)} ] $r(t)>X(x_0,t_0;t-t_0)$ for all $t>t_0$ and $r(t)<X(x_0,t_0;t-t_0)$ for all $t\in [0,t_0)$.
\end{itemize}
Moreover, if the initial data satisfies
\begin{equation}\label{1.6}
    u_0(x)\geq c (r(0)-x)_+^\gamma \quad\text{ for some }c>0,\gamma \in (0,2) \text{ and all }|x-r(0)|<c,
\end{equation}
then {\normalfont{(ii)}} happens for all boundary points.
\end{proposition}


The proposition is a direct corollary of \cite[Theorem 1.2]{kim2018fb}, and we will sketch the proof in the appendix.

The alternative (i) corresponds to a relative waiting time phenomena. The second alternative shows that the support $\{u(\cdot,t)>0\}$ is always strictly expanding relatively to the streamlines once it starts strictly relatively expanding. The growth condition \eqref{1.6} prevents waiting times and forces the free boundary to start moving (relatively to the streamlines) immediately.  
The condition \eqref{1.6} is optimal in the sense that there are stationary solutions to \eqref{main} with quadratic growth at the free boundary (see the proof of Theorem \ref{T.1.3}).

As discussed before, we want the free boundary not only to strictly expand but also to have a positive expanding speed (which is the non-degeneracy) since we expect that regularization comes from movement.
In \cite{kim2018fb}, the non-degeneracy for general dimensions is obtained for type (ii) free boundary with an extra monotonicity assumption which is satisfied in a traveling wave type setting. Nevertheless it is not easy to verify the monotonicity assumption in other general situations even in dimension one, and proving non-degeneracy is still a hard problem. We solved this problem in space dimension one by taking a different approach. We are able to improve type (ii) free boundary to non-degenerate free boundary that corresponds to a positive sign of $-D_x^- u(r(t),t)$ by \eqref{nondeg}. For simplicity, we assume  \eqref{1.6}.

\begin{theorem}\label{T.1.2}
{\rm [Lemma~\ref{L.4.2}, Corollary~\ref{C.5.2}]}
Let $u$ be given as in Theorem \ref{T.1.1}. Then $-D_x^- u(r(t),t):=-\lim_{x\to r(t)^-}u_x(x,t)$ exists for all $t>0$. If in addition \eqref{1.6} holds, then   $-D_x^- u(r(t),t)$ is positive for all $t>0$.


\end{theorem}



To prove the non-degeneracy result, we will firstly show that if the free boundary is non-degenerate at one time, then it is non-degenerate for all time after. The corresponding conclusion for $(PME)$ in spatial dimension one is given in \cite{caffarelli1979regularity}.
In \cite{caffarelli1979regularity}, one key step is to use the fundamental estimate and Barenblatt profiles to construct suitable barriers, which then reveal a lower bound on the acceleration of the free boundary. 
For us, additional difficulties come from the drift. 

With drifts, the most direct approach is to follow the streamlines. However on one hand the coordinate of using streamlines does not cope well with the diffusion in the equation \eqref{main}, on the other hand using simple approximations of streamlines might not be accurate enough for the purpose of estimating the second derivative of the free boundary. 
In fact we construct new barriers involving the second order in time approximation of the streamlines, see \eqref{5.5'}, to carry out the argument. We rely heavily on the bound of $\|B\|_{C_x^3{C}_{t,{\loc}}^{1,1}}$, in which the estimate for time derivatives is non-trivial. Then the estimate at each single time, interpreted as inequalities between distributions, implies that a positive relative expanding speed of the free boundary can not decrease to $0$ in finite time. From this we conclude with the non-degeneracy property.


\medskip

With non-degeneracy, we are able to prove the following regularity of the solution which is uniform up to the free boundary. It then follows the regularity of the free boundary.

\begin{theorem}\label{T.1.4}
Assume the conditions of Theorem \ref{T.1.1}. For any $t_0>0$ and $k\in\bbN^+$, there exist $\eta,C>0$ such that 
\[\|u\|_{C_{x}^k{C}_t^{2,1}(N_{\eta}(t_0))}\leq C\]
where
\begin{equation}\label{Neta}
    N_{\eta}(t_0):={\{(x,t)\,|\, r(t)-x\in (0,\eta),\,|t-t_0|<\eta\}},
\end{equation}
and $r(\cdot)$ is a $C^{2,1}$ function on $(t_0-\eta,t_0+\eta)$.

\end{theorem}

The proof of the theorem is given in Section \ref{s.5}. We adopt the inductive argument given in \cite{LV12} where, after assuming non-degeneracy, both the solution and the free boundary of $(PME)$ are shown to be smooth. In our case, a notable modification is necessary. To compensate the effect from the drift, we will follow the streamline starting at one boundary point and study equation \eqref{6.8}. 

\medskip

Now let us explain (with some formal calculations) the reason why we are only able to obtain $C^{1,1}_{\loc}$ and $C^{2,1}_{\loc}$ regularity in time for $B$ and $u$ respectively, even with smooth $V,W$. 
The problem comes from the nonlocal drift. Since $B=-V-W*\varrho$, to study the regularity of $B$, we need to make sense of the derivatives of $\varrho$ in time.
Recall that $\varrho=\l\frac{m-1}{m} u \r^\frac{1}{m-1}$. In general if $\frac{1}{m-1}\notin \bbN$, $\varrho$ is not a smooth function even when $u$ is smooth. Suppose that $u_t>0$ on the free boundary and then $\varrho_t$ is unbounded near the free boundary. As discussed before, it can be shown that $\varrho_t$ is integrable in space. Formally the equation \eqref{main} yields
\begin{equation}
    \label{1.8}\varrho_{tt}=(\varrho^m)_{xxt}-(B\varrho)_{xt}. 
\end{equation}
Let us ignore functions with only spatial derivatives, and then the right-hand side of \eqref{1.8} is of the linear form $f_1\varrho_t+f_2\varrho_{xt}+f_3\varrho_{xxt}$ for some bounded functions $f_i$.
Thus, we expect $\varrho_{tt}$ to be integrable. However $\varrho_{ttt}$ can be very singular near the free boundary. In fact further differentiating \eqref{1.8} in $t$ shows 
\begin{equation}
    \label{1.9}\varrho_{ttt}=(\varrho^m)_{xxtt}-(B\varrho)_{xtt}
\end{equation}
and there are singular terms involving $\varrho_t^2$ in the expression due to the nonlinearity of the equation. Also formally computing the convolution $W*\varrho_{ttt}$ using \eqref{1.9} and integration by parts, we might need to evaluate singular function $(\varrho^m)_{tt}$ at free boundary points. Thus it is possible that $B_{ttt}$ is not well-defined and so we are only able to obtain $C^{1,1}_{\loc}$ regularity of $B$ in time. 
Due to this and in view of the pressure variable equation, we can only expect $u$ to be $C^{2,1}_{\loc}$.

In general, suppose $W\neq 0$, $\frac{1}{m-1}\notin \bbN$ and there is no topological changes of the support, the problem of smoothness of both $B$ and $u$ in time remains open.



\medskip

When the initial data's support is one open interval, the support of the solution is simply an evolving interval. Then we have the following global result, in which we will use the notations
\[
\Omega:=\{(x,t)\,|\,u(x,t)>0,t>0\},\quad \Omega_t:=\{ u(\cdot,t)>0\}.
\]

\begin{theorem}\label{T.1.5}
Assume the conditions of Theorem \ref{T.1.1}. Suppose $\Omega_0$ is a finite interval and 
\beq\lb{1.91}
u_0(x)\geq c (\text{dist}(x,(\Omega_0)^c)^\gamma \quad\text{ for some }c>0,\gamma \in (0,2).
\eeq
Then in $\Omega$, $u$ is smooth in space and $C^{2,1}_{\loc}$ in time locally uniformly up to the free boundary (in the sense of Theorem \ref{T.1.4}), and there exist two $C^{2,1}_{\loc}$ functions $l(\cdot),r(\cdot)$ on $(0,\infty)$ such that $\Omega_t=(l(t),r(t))$.

Moreover if $W\equiv 0$ or $\frac{1}{m-1}\in\bbN$, $u$ is smooth in $\Omega$ locally uniformly up to the free boundary, and there exist two smooth functions $l(\cdot),r(\cdot)$ on $(0,\infty)$ such that $\Omega_t=(l(t),r(t))$.
\end{theorem}

Here the condition \eqref{1.91} corresponds to \eqref{1.6} which forces both the right and left free boundary to move relatively to the streamlines at all time. Thus instead of \eqref{1.91}, if assuming that $l(\cdot),r(\cdot)$ are not of type (i) after time $t_0\geq 0$, then the same conclusion holds for all $t>t_0$. The proof is the same.

For $(PME)$, the free boundary is strictly expanding after a finite time, \cite{book}. However for the equation with drifts, both nonlocal and local ones, it is possible to have permanent waiting time i.e. there is one streamline lying on the free boundary for all time. 

\begin{theorem}\label{T.1.3}
Let $\bbT$ denote the one dimensional flat torus.
There exist $W\in C^\infty(\mathbb{T})$ and a non-negative function $\varrho\in C^{\infty}(\mathbb{T})$ such that $\{\varrho(x)=0\}\neq \emptyset$ and $\varrho$ is a stationary solution to \eqref{main} with $m=2$, $V\equiv 0$ and domain $\mathbb{T}$.

For any $m>1$ and an open set $I\subseteq\bbR$, there exist a non-negative function $\rho_0$ with support $I$ and a smooth $V=V(x)$ such that the solution $\rho$ to \eqref{main} with $W\equiv 0$ and with initial data $\varrho(0)$ satisfies:
the free boundary $\Gamma_t=\Gamma_0$ is time-independent.

\end{theorem}

The proof is given at end of Section \ref{s.5}.
In the above two examples, $B= 0$ on the time-independent free boundary. Then streamlines starting at those points are stationary. Also, the solutions have quadratic growth near the free boundary. This implies that the growth condition \eqref{1.6} is indeed necessary to have strictly expanding (relatively to streamlines) free boundary. 

Let us mention that when $W\equiv 0$, several examples in general dimensions are given in \cite{kim2018fb} showing the preservation and formation of singularities at free boundary points with waiting time.

\subsection{Outline of the paper}

Section \ref{s.2} contains notations, preliminary definitions and the proof of the fundamental estimate. Section \ref{s.3} proves the Lipschitz continuity of $u$, and then the $C_x^\infty{C}_t^{1,1}$ regularity of $B$. In \ref{ss.3.1}, we study the equation of the free boundary and prove Darcy's law. Non-degeneracy of the free boundary is given in Section \ref{s.4} under the mild condition \eqref{1.6} on the initial data. In Section \ref{s.5}, we consider higher regularity of the solution and the free boundary, and we provide the proofs of Theorems \ref{T.1.4}-\ref{T.1.3}.


We gratefully acknowledges partial support by an
AMS-Simons Travel Grant. We would like to thank anonymous referees for their helpful comments and suggestions.

\section{Notations and Preliminaries}\label{s.2}

\subsection{Notations.}


Assume $O\subseteq \mathbb{R}\times \mathbb{R}^+$ is open. For $m,n\in\bbN\cup \{\infty\}$, $C_x^m{C}_t^n(O)$ denotes functions in $O$ that are $m$-times continuously differentiable in space and $n$-times continuously differentiable in time.  Let $f(x,t)$ be a function on $\mathbb{R}\times \mathbb{R}^+$ and suppose $f\in C_x^m{C}_t^n(U)$. For $m,n<\infty$, we write
\[\|f\|_{C_x^m{C}_t^n(O)}:=\sum\limits_{ \substack{
0\leq i\leq m\\
0\leq j\leq n}}\sup_{(x,t)\in O}|\partial_x^i\partial_t^j f(x,t)|.\]
By $C_{x,t}^n$ we mean $C_x^n{C}_t^n$. For $n\geq 1$, we denote $f^{(n)}:=\partial^n_x f$ for abbreviation.

For $0<\gamma\leq 1$, $m\in \bbN\cup\{\infty\}$ and $n\in\bbN$, the H\"{o}lder space $C_x^{m}{C}_t^{n,\gamma}(O)$ consists of all functions $f\in C_x^{m}{C}_t^n(O) $ for which the following norm is finite
\[\|f\|_{C_x^{m}{C}_t^{n,\gamma}(O)}:=\|f\|_{C_x^{m}{C}_t^n(O)}+\sup_{\substack{
    (x,t),(x,s)\in \overline{O}\\
    t\neq s}}\left\{\frac{|\partial_x^m\partial_t^nf(x,t)-\partial_x^m\partial_t^nf(x,s)|}{|t-s|^\gamma}\right\}.\]
If $\gamma=1$, then $\partial_x^m\partial_t^nf$ satisfies the Lipschitz condition in time. By $f\in C_x^mC_{t,\loc}^{n,\gamma}$, we mean that $\|f\|_{C_x^{m}{C}_t^{n,\gamma}(\bbR\times [t_1,t_2])}\leq C$ for some  $C=C(t_1,t_2)>0$ for all $0<t_1<t_2<\infty$.

Let $U$ be an open subset of $\bbR$. 
The H\"{o}lder space $C^{n,\gamma}(U)$ consists of all functions in $U$ that are $n$-times differentiable and  $g^{(n)}$ is $\gamma$-H\"{o}lder continuous. The $C^{n,\gamma}$ norm is defined similarly as the above.

We write $\|f\|_\infty$ as the  essential supremum norm ($L^\infty$ norm) of $f$ in its domain.

By $U^c$, we mean the complement of $U$ in $\mathbb{R}$.

We write
\[\Omega(\varrho):=\{(x,t)\,|\, \varrho(x,t)>0,\,t>0\}, \quad \Omega_t(\varrho):=\{x\,|\, \varrho(\cdot,t)>0\}\]
and
\[ \Gamma_t(\varrho):=\partial \Omega_t(\varrho),\quad \Gamma(\varrho):=\bigcup_{t\in\bbR^+}(\Gamma_t\times\{t\}).\]
Suppose $\varrho$ solves \eqref{main} and $u$ is its pressure. Then the above four sets stay the same after replacing $\varrho$ by $u$. We will omit their dependence on $\varrho$ (or $u$) when it is clear from the context.

Throughout the paper, if there is no further description, we denote $C$ as various \textit{universal constants}, by which we mean constants that only depend on $m,\varrho_0$, and regularities of $V,W$.

\medskip

Now we introduce the notion of solutions. 

\begin{definition}\label{def1.1}
Let $\varrho_0(x)$ be a non-negative, bounded and integrable function. For any $T>0$, we say that a non-negative, bounded and integrable function $\varrho(x,t):\mathbb{R}\times (0,T)\to[0,\infty)$ is a solution to \eqref{main} if 
\begin{equation}\label{d11}
\begin{aligned}
&\quad \varrho\in C([0,T],L^1(\mathbb{R}))\cap L^\infty(\mathbb{R}\times[0,T])\;\text{ and }\; (\varrho^m)_x\in L^2(\mathbb{R}\times [0,T]),\end{aligned}
 \end{equation}
and for all test functions $\phi\in C_c^\infty(\mathbb{R}\times [0,T))$
\[\int_0^T\int_{\mathbb{R}} \varrho\,\phi_t\,dxd= \int_{\mathbb{R}} \varrho_0\,\phi(\cdot,0)\,dx+\int_0^T\int_{\mathbb{R}}(( \varrho^m)_x+\varrho\,(V+W*\varrho))\,\phi_x\, dxdt.\]

We say that $u=\frac{m}{m-1}\varrho^{m-1}$ is the pressure of the density $\varrho$.
\end{definition}

\begin{theorem}\label{exitence boundedness}
Assume that $\varrho_0 $ is non-negative, bounded and compactly supported, and $V,W\in C^1_{x,t}$.
Then there exists a unique weak solution $\varrho$ to \eqref{main} with initial data $\varrho_0$ for all time. Moreover there exists a constant $C>0$ depending only on $m,\|\varrho_0\|_\infty+\|\varrho_0\|_{L^1},\|V\|_{C^0_{x,t}},\|W\|_{C^0_{x,t}}$ such that 
$
\|\varrho\|_{L^\infty(\bbR\times\bbR^+)}\leq C.
$
\end{theorem}

The well-posedness result can be derived similarly as done in \cite{bertozzi2009existence,bedrossian2011local}. 
Since the solution to be derived is compactly supported due to that the initial data is compactly supported (see Lemma \ref{L.2.2}), we can take advantage of space dimension one and modify $W$ for $|x|$ being large enough so that it is the spatial gradient of an admissible kernel in the sense of Theorem 4.2 in \cite{bertozzi2009existence} and the modification does not affect the solution in a finite time. Then we are able to adapt the
regularisation technique (used in the paper to remove the degeneracy of the diffusion) to obtain the unique weak solution that is compactly supported in the finite time. If $W$ is a potential of a convex function, we also refer readers to Theorem 11.2.8 \cite{ambrosio2008gradient} where the gradient flow structure is employed. Uniform boundedness of the solution for all time is obtained in \cite{blanchet2009critical,bertozzi2009existence,zhang2018class}.

\medskip

Next we introduce the fundamental estimate.

\begin{proposition}{\rm [The fundamental estimate]} \label{T.2.3}
Let $\varrho$ be from Theorem \ref{exitence boundedness} and let $u$ be its pressure. Suppose $V,W\in C_x^3C^0_t$.
We have for some $C>0$ only depending on $m,\|\varrho_0\|_\infty+\|\varrho_0\|_{L^1}, \|V\|_{C_x^3C^0_t}$ and $\|W\|_{C_x^3C^0_t}$ such that for all $t>0$,
\[ u_{xx} (\cdot,t)\geq -\frac{1}{t}-C\quad \text{ in }\mathbb{R}\] 
in the sense of distribution.
\end{proposition}
The proof of the theorem is similar to the one of Theorem 1.1 \cite{kim2018fb} where the case of $W\equiv 0$ is studied. We include the proof in the appendix.

\medskip

With the existence of the solution and due to Theorem \ref{T.1.1}, we can view $B=-V-W*\varrho$ as a known Lipschitz function of $(x,t)$. Then solving the following local-drift equation
\begin{equation}
    \label{1.1}
    \varrho_t=(\varrho^m)_{xx}-(\varrho \,B)_x
\end{equation}
with initial data $\varrho_0$ is the same as solving for \eqref{main}. Setting $u=\frac{m}{m-1}\varrho^{m-1}$, the corresponding pressure variable equation of \eqref{1.1} is
\begin{equation}\label{1.2}
\begin{aligned}
u_t=(m-1) u\, u_{xx}+| u_x|^2- u_x{B}-(m-1)u \,{B}_x\quad \text{in $\mathbb{R}\times \mathbb{R}^+$.}
\end{aligned}
\end{equation}

Below we give the notions of weak solutions to \eqref{1.1} and \eqref{1.2} which are similarly to those in Definition \ref{def1.1} (with $V,W$ replaced by $B,0$ respectively). However for the equation with local drift, comparison principle is available and so we also introduce sub/super solutions. 

\begin{definition}
Let $T,\varrho_0$ be as in Definition \ref{def1.1} and suppose $\varrho$ satisfies \eqref{d11}. Let $B$ be a bounded vector field in $\bbR\times [0,\infty)$. We say that $\varrho$ is a subsolution (resp. supersolution) to \eqref{1.1} with initial data $\varrho_0$ if
\begin{equation}
    \label{def sol2}
    \int_0^T\int_{\mathbb{R}} \varrho\,\phi_tdxdt\geq (\text{resp. } \leq) \int_{\mathbb{R}} \varrho_0(x)\phi(0,x)dx+\int_0^T\int_{\mathbb{R}}(( \varrho^m)_x-\varrho\,B)\phi_x\; dxdt,
    \end{equation}
     for all non-negative  $\phi\in C_c^\infty(\mathbb{R}\times [0,T))$.

We say that $\varrho$ is a weak solution to \eqref{1.1} if it is both sub- and supersolution of \eqref{1.1}. We also say that $u:= \frac{m}{m-1}\varrho^{m-1}$ is a {\it solution} ({resp.} {\it super/sub solution}) to \eqref{1.2} with known $B$, if $\varrho$ is a weak solution ({resp.} super/sub solution) to \eqref{1.1}.
\end{definition}
The existence result of \eqref{1.1} can be find in \cite{alt1983quasilinear} if $B$ is smooth, and in \cite{zhang2017regularity} if $B$ is bounded.

\medskip


We will make use of the following comparison principle for \eqref{1.1}.

\begin{theorem}{\rm [Theorem 2.2, \cite{alt1983quasilinear}]}\label{thm:comp}
Suppose $U$ is an open subset of $\mathbb{R}$ and $B\in C^{1}_{x}C^0_t(\bbR\times[0,\infty))$. For some $T>0$ let $\underline{\varrho},\bar{\varrho}$ be respectively a subsolution and a supersolution of \eqref{1.1} in $U\times (0,T)$ such that $ \underline{\varrho}\leq \bar{\varrho}$ a.e. on the parabolic boundary of $U\times (0,T)$. Then $\underline{\varrho}\leq \bar{\varrho} $ in $U\times(0,T)$.
\end{theorem}

The following technical lemma is useful when we apply the comparison principle.
\begin{lemma}{\rm [Lemma 2.6, \cite{kim2018fb}]}\label{L.2.4}
Under the conditions of Theorem \ref{thm:comp}, let $\psi$ be a non-negative continuous function defined in $U \times [0,T]$ such that
\begin{itemize}
    \item[(a)] $\psi$ is smooth in its positive set and in this set we have $\psi_t- (\psi^m)_{xx}-(\psi\, B)_x\geq 0$,
    \item[(b)] $\psi^\alpha$ is Lipschitz continuous for some $\alpha\in (0,m)$,
    \item[(c)] $\Gamma(\psi)$ has Hausdorff dimension $1$.
\end{itemize}
Then 
\[\psi_t-(\psi^m)_{xx}-(\psi\,B)_x\geq 0 \text{ in }U\times (0,T)\]
in the weak sense.
\end{lemma}

To end this section, let us quantify the finite propagation property of the drift equation. The proof makes use of the comparison principle and it is postponed to the appendix.

\begin{lemma}\label{L.2.2}
Suppose $u$ solves \eqref{1.2} with vector field $B\in C_x^1C^0_t(\mathbb{R}\times [0,\infty))$, and with non-negative, bounded and compactly supported initial data $u_0$. Then there exists $C>0$ such that for each $t>0$, $u(\cdot,t)$ is supported in $B_{C(1+t)}$.
\end{lemma}


\section{Pressure's regularity and Darcy's law}\label{s.3}

In this section, we establish some regularity property of the pressure variable $u$ and the drift $B=-\l V+W*\varrho \r(x,t)$. As a corollary we will obtain Darcy's law.

\begin{lemma}\label{L.3.1}
Let $\varrho$ be the solution to \eqref{main} with bounded, non-negative and integrable initial data $\varrho_0$, and suppose that $V,W$ are bounded in $C_x^3C_t^0(\mathbb{R}\times [0,\infty))$. Let $u$ be its pressure variable. Then there is a universal constant $C>0$ such that for any $t>0$ the following hold a.e. in $\mathbb{R}$,
\begin{align*}
    |u_{x}|^2(\cdot,t)+|u_t|(\cdot,t) \leq C(1+\frac1t).
\end{align*}
Here $C$ only depends on $m$, $\|\varrho_0\|_1+\|\varrho_0\|_\infty$ and $ \|V\|_{C_x^3C_t^0(\mathbb{R}\times [0,\infty))}+\|W\|_{C_x^3C_t^0(\mathbb{R}\times [0,\infty))}$.
\end{lemma}
\begin{proof}
Since $\|\varrho(\cdot,t)\|_{L^1(\bbR)}=\|\varrho_0\|_{L^1(\bbR)}$ by the equation, the assumption yields that $B=B(x,t)$ is uniformly bounded in $C^3_xC^0_t(\mathbb{R}\times [0,\infty))$. 
With this vector field, let $u$ be a solution to \eqref{1.2} with initial $u_{0}:=\frac{m}{m-1}\varrho_0^{m-1}$. By taking a sequence of smooth approximations $u^\tau$ with $\tau>0$ of $u$ as described in the proof of Proposition \ref{T.2.3} ($u^\tau\to u$ locally uniformly in $\bbR\times [0,\infty)$ as $\tau\to 0$) and due to this proposition, we can assume that for some universal constant $C>0$, 
\beq\lb{3.10}
|u^\tau_{xx}(\cdot,0)|\leq \frac{1}{\tau}\text{ in }\bbR,\quad\text{  and }\quad u^\tau_{xx}\geq -\frac{1}{t+\tau}+C\text{ in }\bbR\times [0,\infty).
\eeq
Below we write $u$ in place of $u^\tau$ for the sake of simplicity.


To show the Lipschitz bound in space, we apply Lemma 15.2 \cite{book} which says that if a function $f\in C^2(\bbR)$ satisfies $0\leq f\leq N$ and $f_{xx}>-C$, then $|f_x|\leq \sqrt{2NC}$. In view of the fundamental estimate and Theorem \ref{exitence boundedness}, we obtain
\beq\lb{3.11}
|u_x|^2(x,t)\leq C(1+\frac{1}{t+\tau})\quad \text{ for all }(x,t)\in \bbR\times [0,\infty).
\eeq
The lower bound in the second estimate follows immediately from the equation \eqref{1.2} and the fundamental estimate. Indeed
\begin{align*}
    u_t=(m-1)u\, u_{xx}+|u_x|^2-u_x B-(m-1)u B_x\geq -C(1+\frac{1}{t+\tau}).
\end{align*}

 For the upper bound, we modify the proof of Theorem 15.5 \cite{book} (where $(PME)$ is considered). 
Set
\[\varphi:=u_t+(m-1)u_x^2+u_xB+(m-1)uB_x=(m-1)u\, u_{xx}+m u_x^2.\]
By direct computations,
\begin{align*}
    \varphi_x &=(m-1)u u_{xxx}+(3m-1)u_x u_{xx},\\
    \varphi_{xx} &= u_{txx}+2(m-1)(u_{xx}^2+u_x u_{xxx})+(u_{xxx}B+(m+1)u_{xx}B_x+(2m-1)u_xB_{xx}+(m-1)uB_{xxx}),\\
    \varphi_t &=(m-1)u u_{xxt}+(m-1)u_t u_{xx}+2m u_{xt}u_x\\
    &=(m-1)u u_{xxt}+(m-1) u_{t}u_{xx}+2m(m+1)u_x^2u_{xx}+2m(m-1)uu_xu_{xxx}\\
    &\quad -2m(u_x u_{xx} B+m u_x^2 B_x+(m-1)u u_x B_{xx}).
\end{align*}

Next define $\calL_1(\varphi):=\varphi_t-(m-1)u \varphi_{xx}-(2u_x -B)\varphi_x$. We get
\begin{align*}
\calL_1(\varphi)&= (m-1) u_{t}u_{xx}+2m(m+1)u_x^2u_{xx}+2m(m-1)uu_xu_{xxx}\\
    &\quad -2m(u_x u_{xx} B+m u_x^2 B_x+(m-1)u u_x B_{xx})  \\
&\quad -(m-1)u (2(m-1)(u_{xx}^2+u_x u_{xxx})+((m+1)u_{xx}B_x+(2m-1)u_xB_{xx}+\\
&\quad \quad (m-1)uB_{xxx})) -(2u_x-B)(3m-1)u_x u_{xx}\\
&= (m-1)u_{xx} (u_t+2(m-1)u_x^2-2(m-1)u\, u_{xx})\\
&\quad -(m-1)u ((m+1)u_{xx}B_x+(4m-1)u_x B_{xx}+(m-1)uB_{xxx})\\
&\quad -2m^2 u^2_x  B_x+{2(4m-1)}u_xu_{xx}B .
\end{align*}
Using the equation of $u$, the above
\begin{align*}
&=(m-1)u_{xx} (-(m-1)u\, u_{xx}+(2m-1)u_x^2+\frac{7m-1}{m-1}u_xB-2m uB_x)\\
&\quad -(m-1)u ((4m-1)u_x B_{xx}+(m-1)uB_{xxx})-2m^2 u^2_x  B_x. 
\end{align*}
Then apply $(m-1)u_{xx}=\frac{\varphi-m u_x^2}{u}$ to obtain
\begin{align*}
\calL_1(\varphi)
&=-\frac{1}{u}(\varphi-m u_x^2)(\varphi-(3m-1)u_x^2-\frac{7m-1}{m-1}u_xB+2m uB_x)\\
&\quad -(m-1)(4m-1)u u_xB_{xx}-(m-1)^2u^2 B_{xxx}-2m^2 u_x^2 B_x\\
&=: -\frac{1}{u}(\varphi-f_1(u_x))(\varphi-f_2(u,u_x,B))
+f_3(u,u_x,B)
\end{align*}
where
\begin{align*}
   &f_1(u_x)=m u_x^2,\quad \quad f_2(u,u_x,B)= (3m-1)u_x^2+\frac{7m-1}{m-1}u_xB-2m uB_x,\\
   &f_3(u,u_x,B)=-(m-1)(4m-1)u u_xB_{xx}-(m-1)^2u^2 B_{xxx}-2m^2 u_x^2 B_x.
\end{align*}
Using that $B$ is uniformly bounded in $ C^3_xC_t^0(\mathbb{R}\times [0,\infty))$, $|u|\leq C$, and $|u_x(\cdot,t)|^2\leq C(1+\frac{1}{t+\tau})$, it follows that for some universal $C_1>0$,
\[ \|f_1\|_\infty+\|f_2\|_\infty+\|f_3\|_\infty\leq C_1(1+\frac{1}{t+\tau}).\]
From the above computations, we find $\calL_1(\varphi)-F(\varphi)=0$ where
\[F(\varphi):=-\frac{1}{u}(\varphi-f_1)(\varphi-f_2)
+f_3.\]

Now take $w(x,t):=C_2(1+\frac{1}{t+\tau})$ with $C_2>C_1$.
Then 
\begin{align*}
    \calL_1(w)-F(w)&=-\frac{C_2}{(t+\tau)^2}+\frac{1}{u}(C_2(1+\frac{1}{t+\tau})-f_1)(C_2(1+\frac{1}{t+\tau})-f_2)-f_3\\
    &\geq -\frac{C_2}{(t+\tau)^2}+\frac{1}{C}((C_2-C_1)(1+\frac{1}{t+\tau}))^2-C_1(1+\frac{1}{t+\tau})\geq 0
\end{align*}
for all $t\geq 0$ if $C_2$ is large enough depending only on $C,C_1$. Note that \eqref{3.10} and \eqref{3.11} yield
$
\varphi(x,0)\leq \frac{C}\tau.
$
Hence by further taking $C_2$ to be large enough, we have $w\geq \varphi$ at $t=0$. By comparison principle for the parabolic operator $(\calL_1-F)(\cdot)$, we conclude that $\varphi\leq C_2(1+\frac1{t+\tau})$. Then the definition of $\varphi$ yields
\begin{align*}
    u_t=\varphi-(m-1)u_x^2-u_xB-(m-1)uB_x\leq C(1+\frac{1}{t+\tau})
\end{align*}
for some universal $C>0$. Finally, taking $\tau\to 0$ finishes the proof.
\end{proof}

Viewing $B$ as a given vector filed of $(x,t)$, 
below we show that $B=B(x,t)$ is smooth in space and $C^{1,1}$ in time.

\begin{lemma}\label{L.3.2}
Let $\varrho$ be the solution to \eqref{main} with bounded, non-negative, compactly supported initial data $\varrho_0$, and suppose that $V,W$ are bounded in $ C_{x,t}^\infty(\mathbb{R}\times [0,\infty))$. Let $u$ be its pressure variable. Then for each $k\geq 0$, there exists a constant $C_k>0$ such that
\[\|B\|_{C_x^kC^0_t(\bbR\times [0,\infty))}\leq C_k,\]
and for any $T\geq 1$ and $t\in (0,T]$,
\begin{equation}
    \label{l.3.2}
\|{B}_{t}(\cdot,t)\|_{C_x^k(\bbR)}+\|{B}_{tt}(\cdot,t)\|_{C_x^k(\bbR)}\leq C_kT\l 1+\frac{1}{t}\r.
\end{equation}
\end{lemma}
\begin{proof}
From the equation \eqref{main}, $\|\varrho(\cdot,t)\|_{L^1}=\|\varrho_0(\cdot)\|_{L^1}<\infty$. Since $B=-V-W*\varrho$ and $V,W$ are smooth in space, $B$ is also smooth in space.

For the regularity of $B_t$, we start with the case of $m\leq 2$. Since 
\[\varrho_t=\frac{1}{m}\l\frac{m-1}{m}u \r^\frac{2-m}{m-1} u_t,\]
Lemma \ref{L.3.1} and boundedness of the solution imply that for some $C>0$,
\begin{equation}\label{3.9}
    |\varrho_t(x,t)|\leq C(1+\frac1t)\quad\text{ in }\bbR\times (0,\infty).
\end{equation}
Due to Lemma \ref{L.2.2},  for all $t\in (0,T]$, $\varrho(\cdot,t)$ is supported in $B_{CT}$ for some $C>0$.
Thus using smoothness of $V$ and $W$, we get
\begin{align*}
    \|B_t(\cdot,t)\|_{C^k_x(\bbR)}&=\|(V_t+W_t*\varrho+W*\varrho_t)(\cdot,t)\|_{C^k_x(\bbR)}\\
    &\leq C(k)(1+\sup_{x\in\bbR}\int_{|r|\leq CT}|\varrho_t(x+r)|dr)\leq C(k)T(1+\frac1t).
\end{align*}

Next, to show $B_t$ is Lipschitz in time, it suffices to show that $W*\varrho_t$ is Lipschitz in time. For a.e.  $0<s<t<T$, and $x\in\bbR$,
\begin{align*}
    &\quad |(W*\varrho_t)(x,t)-(W*\varrho_t)(x,s)|\\
    &\leq |W(\cdot,t)*\varrho_t(\cdot,t)-W(\cdot,s)*\varrho_t(\cdot,t)|(x)+|W(\cdot,s)*\varrho_t(\cdot,t)-W(\cdot,s)*\varrho_t(\cdot,s)|(x)\\
   & \leq C(t-s)1*|\varrho_t(\cdot,t)|(x)+|W(\cdot,s)*\varrho_t(\cdot,t)-W(\cdot,s)*\varrho_t(\cdot,s)|(x)=:A_1+A_2.
\end{align*}
Since $\varrho_t$ is supported in $B_{CT}$, we use \eqref{3.9} to get
\[
A_1\leq CT(1+\frac{1}{s})(t-s).
\]
As for $A_2$, because $\varrho_t(\cdot,t)$ with $t>0$ is an almost everywhere well-defined $L^\infty$ function, \eqref{1.1} yields
\[A_2=\left|\int_\bbR W(x-y,s)\Big[((\varrho^m)_{xx}-(\varrho B)_x)(y,t)-((\varrho^m)_{xx}-(\varrho B)_x)(y,s)\Big]\, dy\right|.\]
Since $\varrho$ is supported in $B_{CT}$, we can change the domain of the integration to $B_{CT}$.
Notice that due to $|u_x|<\infty$, $(\varrho^m)_x=C_m u^\frac{1}{m-1}u_x=0$ on the free boundary, and so applying integration by parts, we get
\begin{align*}
A_2&=\left|\int_{B_{CT}} W_x(x-y,s)\Big[((\varrho^m)_{x}-(\varrho B))(y,t)-((\varrho^m)_{x}-(\varrho B))(y,s)\Big]\, dy\right|\\
&\leq\left|\int_{B_{CT}} W_{xx}(x-y,s)(\varrho^m(y,t)-\varrho^m(y,s))\, dy\right|+\left|\int_{B_{CT}} W_x(x-y,s)((\varrho B)(y,t)-(\varrho B)(y,s))\, dy\right|\\
    &= m\left|\int_{B_{CT}}\int^t_s W_{xx}(x-y,s)\varrho^{m-1}(y,\tau)\varrho_t(y,\tau)\,d\tau dy\right| +\left|\int_{B_{CT}}\int^t_s W_{x}(x-y,s)(\varrho B)_t(y,\tau)\,d\tau dy\right|\\
    &\leq CT(t-s)(1+\frac{1}{s}),
\end{align*}
where in the last inequality we used \eqref{3.9}, and $\varrho, |W_x|, |W_{xx}|,|B|\leq C$ for some universal $C>0$, and $|B_t|\leq CT(1+\frac{1}{s})$ in $\bbR\times [s,T]$.


Overall, we find that there exists $\beta(\cdot,t)=B_t(\cdot,t)$ for a.e. $t\in (0,T)$ such that $\beta(\cdot,\cdot)$ is Lipschitz continuous in time with bound $CT\left(1+\frac{1}{t}\right)$ in $\bbR\times [t,T]$, and $\beta,B$ are smooth in space. Using \eqref{3.9} and the definition of $B$ yields that $B(x,t)=\int_1^t\beta(x,s)ds+B(x,1)$ is continuously differentiable (in both $x$ and $t$) for all $t>0$. Since $\beta$ is Lipschitz in time, we obtain Lipschitz continuity of $B_t$ in time.
Similarly we can get for all $k\geq 0$,
\[\|{B}_{tt}(\cdot,t)\|_{C^k_x(\bbR)}\leq C_kT(1+\frac{1}{t}).\]


When $m>2$, again by Lemma \ref{L.3.1}, in the positive set of $\varrho$,
\beq\lb{3.92}
\begin{aligned}
    \varrho_t=\frac{1}{m}\l\frac{m-1}{m}u \r^\frac{2-m}{m-1} u_t\leq C\left(1+\frac{1}{t}\right)\varrho^{{2-m}}\quad\text{  and }\quad |\varrho_x|\leq C(1+\frac1t)^{1/2}\varrho^{2-m}.
\end{aligned}
\eeq
So we deduce
\[
(\varrho^m)_{xx}=\varrho_t+(B\varrho)_x\leq C(1+\frac1t)\varrho^{2-m}\quad\text{ in }\{\varrho>0\}.
\]
Therefore the bounded non-negative continuous function $w(\cdot,t):=(1+\frac1t)^{-{\gamma}}\varrho^m(\cdot,t)$ with ${\gamma}:=\frac{m}{2(m-1)}$ (then $\gamma\in (\frac12,1)$ by $m>2$) satisfies for some universal $C>0$,
\begin{equation}\label{3.4}
\begin{aligned}
(w)_{xx}&=(1+\frac1t)^{-{\gamma}}(\varrho^m)_{xx}\leq C(1+\frac1t)^{1-{\gamma}}\varrho^{2-m}\leq  Cw^{-\frac{m-2}{m}}\quad \text{ in }\{w>0\}.
\end{aligned}
\end{equation}
It follows from the fundamental estimate that
\begin{equation}
    \label{3.3}
    (w^{\frac{m-1}{m}})_{xx}=(1+\frac1t)^{-\frac{(m-1){\gamma}}{m}}(\varrho^{{m-1}})_{xx}\geq 
    -C\quad \text{ in } \quad \calD'(\bbR).
\end{equation}
With properties \eqref{3.4} and \eqref{3.3}, Lemma 15.7 \cite{book} implies that
$w_{xx}$ is bounded in $L_{\text{loc}}^p(\bbR)$ for any $p\in [1,\frac{m-1}{m-2})$. 
Moreover the $L_{\text{loc}}^p(\bbR)$ bound is independent of $t$ and the location: 
\begin{equation}\label{3.5}
    \int_{a-1}^{a+1}|w_{xx}(x,t)|^p dx\quad\text{ is uniformly bounded for all $a\in\bbR$, and $t>0$.}
\end{equation}
This and \eqref{main} yield
\begin{equation}\label{3.8}
    Z:=(1+\frac1t)^{-{\gamma}} (\varrho_t+B\varrho_x)=(w)_{xx}-(1+\frac1t)^{-{\gamma}} B_x\varrho\in L^\infty((0,\infty),L^p_{\text{loc}}(\bbR)).
\end{equation}

Now for any fixed $t_0>0$, we use the substitution of streamlines
$
Y:=Y(t)=Y(y,t_0;t-t_0)$ (then $y\to Y(y,t_0;t-t_0)$ is a bijection when $|t-t_0|$ is small enough). With this coordinate, due to $Y_t(t)=B(Y(t),t)$, we have
\beq\lb{3.88}
\partial_t \varrho(Y(t),t) =\varrho_t(Y(t),t)+ (B \varrho_x)(Y(t),t)=(1+\frac1t)^\gamma Z(Y(t),t).
\eeq
Since $\partial_y\partial_t Y=\partial_Y B(Y,t)\partial_yY$ and $Y_y(t_0)=1$, $\partial_yY$ and $\partial_y\partial_t Y$ are uniformly bounded for all $y\in\bbR$ when $|t-t_0|$ is small enough depending on $\|B\|_{C_x^1C_t^0}$. Moreover, we have
\begin{align}\lb{3.881}
    (W*\varrho)_t(x,t_0)= \partial_t\int_{y\in\bbR} W(x-y,t)\varrho(y,t)dy \big|_{t=t_0} =\partial_t\int_{y\in\bbR} W(x-Y,t)\varrho(Y,t)Y_y dy \big|_{t=t_0}.
\end{align}
Recall that $\varrho(\cdot,t)$ is compactly supported in $B_{CT}$ for $t\leq T$.
Using \eqref{3.5}--\eqref{3.881} with $p=1$, and $|W_x|,|W_t|$, $|Y_t|$ and $|\partial_y\partial_tY|$ are bounded, we obtain
\begin{equation}
    \label{3.6}
    \begin{aligned}
|(W*\varrho)_t(x,t_0)|&\leq C+  \int_{B_{CT}}W(x-Y,t_0) \partial_t\varrho(Y,t_0)Y_y d y  \\
&\leq C+ C(1+\frac1{t_0})^\gamma \int_{B_{CT}}|W(x-y,t_0) Z(y,t_0)|d y  \\
&\leq C+ C(1+\frac{1}{t_0})^{\gamma}\|W(x-\cdot)\|_{\infty}\|Z(\cdot,t_0)\|_{L^1(B_{CT})}\leq CT(1+\frac{1}{t_0})^{\gamma}.
\end{aligned}
\end{equation}
Similarly for any $k\geq 0$, since $W$ is bounded in $C_x^kC_t^1$, we get $\|(W*\varrho)_t(\cdot,t)\|_{C^k_x}\leq C_kT (1+\frac1t)^{\gamma}$ which implies that for $t\in (0,T]$, 
\beq\lb{3.90}
\|B_t(\cdot,t)\|_{C^k_x(\bbR)}\leq C_kT(1+\frac1t)^{\gamma}\leq C_kT(1+\frac1t).
\eeq

Now we are left to show that $\partial_x^k\partial_tB$ is locally Lipschitz in time. Let us only consider the case of $k=0$ and prove the following for all $x\in\bbR$, and $0<t_0<t<T$ such that $t-t_0$ is small enough,
\beq\lb{3.86}
|(W*\varrho)_t(x,t)-(W*\varrho)_t(x,t_0)|\leq CT(t-t_0)(1+\frac{1}{t_0}).
\eeq
Let $Y(t)=Y(y,t_0;t-t_0)$ be as before. Note that \eqref{4.1} yields
\[
\partial_t^2Y(y,t_0;t-t_0)=  B_{Y}(Y,t)Y_t+B_{t}(Y,t)=B_{Y}(Y,t)B(Y,t)+B_{t}(Y,t),
\]
\[
\partial_y\partial_t^2Y(y,t_0;t-t_0)=  B_{YY}(Y,t)Y_yY_t+B_{Yt}(Y,t)Y_y+B_Y(Y,t)Y_{yt}
\]
where $B_{Yt}:=\partial_Y\partial_tB$ and $Y_{yt}:=\partial_y\partial_tY$. 
Thus,  by \eqref{3.90}, $\partial_t^2Y$ and $\partial_y\partial_t^2Y$ are bounded by $CT(1+\frac{1}{t_0})^\gamma$ for all $y\in\bbR$ and $t$ such that $t-t_0$ is small enough. This yields that $\partial_t^2 W(x-Y,t)$ is also bounded by $CT(1+\frac{1}{t_0})^\gamma$ for these $t$.
Then, arguing similarly as before and using \eqref{3.88}, the left-hand side of \eqref{3.86} satisfies
\begin{align*}
    &=\left|\partial_t\left(\int_{y\in\bbR} W(x-Y,\cdot)\varrho(Y,\cdot) Y_ydy\right) (t)-\partial_t\left(\int_{y\in\bbR} W(x-Y,\cdot)\varrho(Y,\cdot) Y_y dy\right)(t_0)\right|\\
    &\leq CT(t-t_0)(1+\frac{1}{t_0})^\gamma +\left|\left(\int_{y\in\bbR} W*(\varrho^m)_{xx}dy\right) (x,t)-\left(\int_{y\in\bbR} W*(\varrho^m)_{xx} dy\right)(x,t_0)\right|\\
    &\qquad\qquad +\underbrace{\left|\left(\int_{y\in\bbR} W*(B_x\varrho)dy\right) (x,t)-\left(\int_{y\in\bbR} W*(B_x\varrho) dy\right)(x,t_0)\right|}_{A_3:=}\\
    &\leq CT(t-t_0)(1+\frac{1}{t_0})^\gamma +\left|\left(\int_{y\in\bbR} W_{xx}*\varrho^m dy\right) (x,t)-\left(\int_{y\in\bbR} W_{xx}*\varrho^m dy\right)(x,t_0)\right|+A_3\\
    &\leq CT (t-t_0) (1+\frac{1}{t_0}),
\end{align*}
where we used $(\varrho^m)_x=\varrho^m=0$ on the free boundary, $\gamma<1$, and $|(\varrho^m)_t|\leq C(1+\frac1{t_0})$ in $\bbR\times [t_0,\infty)$ by \eqref{3.92}. In the last inequality, by \eqref{3.90}, $A_3\leq CT(t-t_0) (1+\frac{1}{t_0})^\gamma$ follows in the same way as we derive \eqref{3.6}.  We finished the proof of \eqref{3.86}.

\end{proof}

\begin{remark}\label{R.3.3}
The restriction of compactly supported solutions can be removed if we assume that the solution is uniformly bounded and $W$ satisfies the following integrability condition:
\begin{itemize}
    \item[(I)] If $m\leq 2$, $W^{j}(\cdot,t)\in L^1(\bbR)$ for all $t\geq 0$ and $j\in\bbN$; 
    if $m>2$, for some $q>m-1$, and for each $j\in\bbN$,
    \[
    \sum_{k\in\bbN}\left( \|W^{(j)}(\cdot,t)\|_{L^q([k,k+1])}+\|W^{(j)}_t(\cdot,t)\|_{L^q([k,k+1])}\right)<\infty,\]
locally uniformly in $t\in [0,\infty)$.
\end{itemize}
Under condition (I), \eqref{l.3.2} holds for some constant $C=C(T)$.
The proof is in the same spirit by making use of \eqref{3.8} for $p\in [1,\frac{m-1}{m-2})$. However the computations are more complicated, and we skip them. We will only consider solutions with compact support later. 

\end{remark}

\subsection{Equation of the free boundary}\label{ss.3.1}


In this subsection, we introduce Darcy's law and the equation of the free boundary. The discussions are parallel to those in section 15 of \cite{book} for the zero drift case.


The following lemma is the same as Lemma 3.3 \cite{kim2018fb} which mainly says that a streamline cannot leave the support of the solution as time evolves.

\begin{lemma}\label{L.4.1}
Let $u$ be a non-negative solution to \eqref{1.2} with compactly supported initial data.
The set $\bigcup_{t>0}(\overline{\{u(\cdot,t)>0\}}\times \{t\})$ is non-contracting along the streamlines i.e. if $u(x_0,t_0)>0$, then $u(X(x_0,t_0;t),t_0+t)>0$ for all $t\geq 0$.
\end{lemma}

In view of the finite propagation property and the fact that the initial data is compactly supported, we can define the right-hand side free boundary as
\begin{equation}
    \label{d.4.1}
    r(t):=\sup\{x\,|\,\varrho(x,t)>0\}=\sup\{x\,|\,u(x,t)>0\}.
\end{equation}

As derived formally in the introduction, we have Darcy's law: $r'(t)=-u_x(r(t),t)+B(r(t),t)$. We prove it below.

\begin{lemma}\label{L.4.2}
For every $t>0$, the following limits exist
\[D_x^- u(r(t),t)=
\lim_{x\to r(t)^-}u_x(x,t),\quad D_t^+ r(t)=\lim_{h\to 0^+}\frac{1}{h}(r(t+h)-r(t)).\]
Moreover, Darcy's law holds in the form
\begin{equation*}
    D_t^+r(t)=-D_x^-u(r(t),t)+B(r(t),t).
\end{equation*}
In particular since $u_x$ and $B$ are bounded, the free boundary is Lipschitz continuous in time.
\end{lemma}
\begin{proof}

The existence of the two limits mainly follows from the fundamental estimate and the regularity of $B$ established in Section \ref{s.3}. The proof is the same as the one for the zero-drift case. We refer readers to Theorem 7.2 \cite{Knerr} and Theorem 15.19 \cite{book}. 

Now let us show Darcy's law, the proof of which is in the same spirit of the proof of Theorem 15.19 \cite{book} (though we need a slightly different barriers). Take one right-hand side free boundary point $(x_0,t_0)$ with $t_0>0$ and by shifting the coordinates, we assume it is $(0,0)$ i.e. $r(0)=0$. Denote $a=-D_x^-u(0,0)\geq 0$. For any small $\eps>0$, consider the following linear functions
\[L^+_\eps(x,t):=(a+\eps)((a+2\eps)t-x+B(0,0)t)_+,\]
\[L^-_\eps(x,t):=(a-\eps)((a-2\eps)t-x+B(0,0)t)_+.\]

We want to compare $u(x,t)$ with $L^+_\eps$ in a domain of the form $R({\delta,\tau})=\{|x|<\delta, t\in (0,\tau)\}$.
By the definition of $a$, if $\delta$ is small enough, $L^+_\eps(x,0)=(a+\eps)(-x)_+\geq u(x,t)$ on the bottom of $R(\delta,\tau)$ and $u(-\delta,0)<L^+_\eps(-\delta,0)$. By continuity of $u$, there is $\tau>0$ such that $ u(-\delta,t)\leq L^+_\eps(-\delta,t)$ for all $t\in [0,\tau)$. Moreover by continuity of $r(t)$, after further assuming $\tau$ to be small enough, we get $u(\delta,t)=L^+_\eps(\delta,t)=0$ for all $t\in [0,\tau)$. 

Next let us check that $L_\eps^+$ is a supersolution to \eqref{1.2}. Indeed in the positive set of $L_\eps^+$, we have
\begin{equation}
    \label{calL}
    \begin{aligned}
\calL(L_\eps^+)&:= (L_\eps^+)_t-(m-1)(L_\eps^+)_{xx}L_\eps^+-|(L_\eps^+)_x|^2+(L_\eps^+)_xB(x,t)+(m-1)L_\eps^+B_x(x,t)\\
    &=(a+\eps)(a+2\eps)+(a+\eps)B(0,0)-(a+\eps)^2-(a+\eps)B+(m-1)L_\eps^+ B_x\\
    &\geq (a+\eps)\eps-C(a+\eps)(|x|+|t|)-C(m-1)L_\eps^+.     
    \end{aligned}
\end{equation}
Here we used the estimate that $B$ is Lipschitz continuous in both space and time.
Now if further letting $\delta,\tau$ to be small enough, we get $|x|+|t|<<\eps$ and $L_\eps^+<<a\eps$, and then $\calL(L_\eps^+)\geq 0$. So $L_\eps^+$ is a supersolution to $\calL$ in $R(\delta,\tau)$. By comparison $u(x,t)\leq L_\eps^+(x,t)$ in $R(\delta,\tau)$. Therefore, the free boundary of $u$ lies to the left of that of $L_\eps^+$ within a short time which implies
\[r(t)\leq r(0)+(a+2\eps)t+B(0,0)t\quad \text{ for }t\leq \tau.\]
We find
\begin{equation}
    \label{e.3}
    D_t^+r(0)=\lim_{h\to 0^+}\frac{1}{h}(r(h)-r(0))\leq a+2\eps+B(0,0).
\end{equation}
After passing $\eps\to 0$, we obtain $D_t^+r(0)\leq a+B(0,0)$.

If $a=0$, it follows from Lemma \ref{L.4.1} and Lipschitz continuity of $B$ that $D_t^+r(0)\geq B(0,0)$ which, combining with \eqref{e.3}, yields the proof. If $a>0$, similarly as done in the above argument, we can show $u\geq L_\eps^-$ in a small neighbourhood of the free boundary point, which implies that
\[D_t^+r(0)=\lim_{h\to 0^+}\frac{1}{h}(r(h)-r(0))\geq a+B(0,0).\]
Overall, we proved for free boundary point $(r(t_0),t_0)$ that $D_t^+r(t_0)=-D_x^-u(r(t_0),t_0)+B(r(t_0),t_0)$.

\end{proof}

\section{Non-degeneracy of the Free Boundary}\label{s.4}

The goal of this section is to prove that if the free boundary is non-degenerate at one time, then non-degeneracy preserves for all time. For this purpose, we only need $V,W\in C_{x}^3C_t^2(\mathbb{R}\times [0,\infty))$.
Throughtout the rest of the paper, let us write the right-hand side free boundary of the solution $u$ (or $\varrho$) as $r(t)$, see \eqref{d.4.1}.


\begin{theorem}\label{T.5.1}
Let $\varrho$ be the solution to \eqref{main} with bounded, non-negative, compactly supported initial data $\varrho_0$, and $V,W\in C_{x}^3C_t^2(\mathbb{R}\times [0,\infty))$. Let $u$ be its pressure variable. Suppose for some $t_0>0$, $-D_x^-u(r(t_0),t_0)>0$. Then for any $T>t_0$ there exists $\sigma=\sigma(t_0,T)>0$ such that
 \[-D_x^-u(r(t),t)\geq -e^{-\sigma (t-t_0)}D_x^-u(r(t_0),t_0)>0\quad \text{ for all }\quad t\in [ t_0,T].\]
\end{theorem}

\begin{proof}

Denote
\[k(t):=-D_x^-u(r(t),t).\] 
Since $u$ is Lipschitz continuous after positive time, we have $k(t)\leq \sigma_0=\sigma_0(t_0)$ for all $t>t_0$.
It follows from Proposition \ref{T.2.3} that $u_{xx}\geq -\sigma_1$ in $\bbR\times [t_0,\infty)$ in the sense of distribution for some $\sigma_1=\sigma_1(t_0)$.
Denote $B(x,t)=-(V+W*\varrho)(x,t)$. Fix any $T>\min\{t_0,1\}$. By Lemma \ref{L.3.2} and the assumption that $V,W\in C_{x}^3C_t^2(\mathbb{R}\times [0,\infty))$, there exists $\sigma_2=\sigma_2(t_0,T)$ such that 
\begin{equation}
    \label{5.1}
    (1+\|B\|_{C_x^3}+\|B_{t}\|_{C_x^2}+\|B_{tt}\|_\infty)^3\leq \sigma_2.
\end{equation}

Take one free boundary point $x_1=r(t_1)$ with $t_1>t_0$.
For simplicity of notations, by performing a translation on $(x,t)$, we can assume $x_1=t_1=0$.

Define 
\begin{equation}\label{L}
    L:=\max\{ (m-1)(5\sigma_3+2\sigma_0), 4\sigma_3\}\quad \text{ and }\quad \lambda(t):=\frac{\sigma_3}{2}e^{-2Lt}
\end{equation}
where $\sigma_3:=\max\{\sigma_1,\sigma_2\}\geq 1$. Next we set $\alpha(t)$ to be the unique solution to
\begin{equation*}
\alpha'(t)=\frac{2 k_0}{\sigma_3}\lambda(t)(1-{L}t)\quad \text{ with }\alpha(0)=\alpha_0:=\frac{k_0}{\sigma_3}\text{ and }k_0:=k(0)>0.
\end{equation*} 

Consider the following barrier
\begin{equation}
    \label{5.5'}
    \underline{u}(x,t)=\lambda\l\alpha(t)^2-\l x+\alpha_0-B t+B\, B_x\frac{t^2}{2}+B_t\frac{t^2}{2}\r^2\r_+.
\end{equation}
It can be checked that,
\[\underline{u}(0,0)=u(0,0)=0,\quad D_x^-\underline{u}(0,0)=D_x^-u(0,0)=-2\lambda(0)\alpha_0=-k_0,\]
and for all $x<r(0)$,
\[ \underline{u}_{xx}(x,0)=-\sigma_3\leq u_{xx}.\]
Hence we have
\[\underline{u}(x,0)\leq u(x,0).\]

We {\bf claim} that $\underline{u}$ is a subsolution to \eqref{1.2} for $t\in [0,\tau_*]$ where 
$\tau_*=\min\{k_0,\tau\}$ for some $\tau>0$ depending only on $t_0,T$ and universal constants.

\medskip

The proof of the {\bf claim} will be given below. We first discuss the consequences.
With the claim, using $\underline{u}(x,0)\leq u(x,0)$ and comparison principle, we obtain $\underline{u}\leq u$ in $\bbR\times [0,\tau_*]$.
By the definition of $\underline{u}$, the right-hand side free boundary (denoted as $\underline{r}=\underline{r}(t)$) of $\underline{u}$ satisfies
\begin{equation}
    \label{5.5}
    \underline{r}(t)=\alpha(t)-\alpha_0+B(\underline{r},t)t-(B(\underline{r},t)B_x(\underline{r},t)+B_t(\underline{r},t))\frac{t^2}{2},
\end{equation}
which is obtained by solving for $\underline{u}(\underline{r}(t),t)=0$.
And so
\begin{equation}
    \label{5.7}
    \underline{r}'(0)=\alpha'(0)+B(0,0)=k_0+B(0,0)=D_t^+r(0).
\end{equation}
Because $\underline{u}\leq u$ in $\mathbb{R}\times [0,\tau_*]$, we know $\underline{r}\leq r$ for $t\in [0,\tau_*]$. 
Hence for $h\leq \tau_*$,
\begin{equation}\label{5.14}
    \underline{r}(h)-\underline{r}(0)-h\underline{r}'(0)\leq r(h)-r(0)-hD_t^+r(0).
\end{equation}

Direct computation yields
\[\alpha''(0)=2\alpha_0 \partial_t(\lambda(t)(1-Lt))|_{t=0}=-6\alpha_0L\lambda(0)=-3L k_0.\]
Recall that $r(0)=\underline{r}(0)=0$. By differentiating \eqref{5.5} twice, we get
\begin{align*}
 \underline{r}''(0)&=-3L k_0+2(B(\underline{r}(0),0))'-B(0,0)B_x(0,0)-B_t(0,0)
\end{align*}
where
\[(B(\underline{r}(0),0))':=\frac{d}{dt}B(\underline{r}(t),t)|_{t=0}.\]
Due to \eqref{5.7},
\[(B(\underline{r}(0),0))'=B_x(0,0)(k_0+B(0,0))+B_t(0,0)=D_t^+B(r(t),t)|_{t=0}.\]
Therefore
\begin{equation}
     \label{5.6}
\begin{aligned}
    \underline{r}''(0)&=-3L k_0+2D_t^+(B(r(t),t))|_{t=0}-B(0,0)B_x(0,0)-B_t(0,0)\\
    &=-3L k_0+D_t^+(B(\underline{r}(t),t))|_{t=0}+B_x(0,0)k_0\\
    &\geq -\sigma k_0+D_t^+(B(r(t),t))|_{t=0}\\
    &:=-\sigma k_0+D_t^+B(r(0),0),
\end{aligned}
\end{equation}
where $\sigma:=3L+\sigma_2$.

\medskip

Now let us go back to any general free boundary point $x=r(t)$ with $t\geq t_0$. According to \eqref{5.14} and \eqref{5.6}, we have for the function
\[g_h(t):=\frac{r(t+h)-r(t)-hD_t^+r(t)}{h^2/2},\]
the estimate
\begin{equation}\label{5.8}
g_h(t)\geq -\sigma k(t)+D_t^+B(r(t),t)+o(h),
\end{equation}
under the condition that $h\leq \min\{k(t),\tau\}$. 

While in the case when $k(t)=0$, by Lemma \ref{L.4.2}, \beq\lb{4.2}
D_t^+r(t)=k(t)+B(r(t),t)=B(r(t),t).
\eeq
Proposition \ref{P.1.2} yields, for all $h>0$, 
\begin{equation}\label{5.15}
r(t+h)>X(r(t),t;h) .   
\end{equation}
Thus
\begin{align*}
    g_h(t)&\geq \frac{2}{h^2}(X(r(t),t;h)-r(t)-hB(r(t),t))\\
   & =\frac{2}{h^2}\left(\int_t^{t+h} B(X(r(t),t;s),t+s)-B(r(t),t)ds\right).
\end{align*}
It follows from \eqref{4.1} that
\[|B(X(r(t),t;s),t+s)-B(r(t),t)|\leq \sigma_2s\leq \sigma_2h,\]
and thus $g_h(t)$ is bounded below by $-2\sigma_2$ when $k(t)=0$.


Due to Lemma \ref{L.4.2}, $r(t)$ is Lipschitz continuous. In view of Lemma \ref{L.3.1}, 
\begin{align*}
    \left|\int_{t_0}^T g_h(t)dt\right|=\left|\frac{2}{h^2}\int_{T}^{T+h}r(t)-r(T)dt-\frac{2}{h^2}\int_{t_0}^{t_0+h}r(t)-r(t_0)dt\right|\leq C(t_0)
\end{align*}
for some $C(t_0)>0$ independent of $h$. Therefore we can select a sequence of $h_n\to 0$ such that $g_{h_n}$ converges to a signed measure $\mu$. In view of the definition of $g_h$ and \eqref{4.2}, we get
\begin{equation}
    \label{5.16}
    \mu=(k(t)+B(r(t),t))'
\end{equation}
in the sense of distribution. 

Denote $E_0:=\{t\in [t_0,T]\,|\,k(t)=0\}$. In view of \eqref{4.2} and \eqref{5.15}, $k(t)$ is the right derivative of a strictly increasing function (for which right derivatives always exist), and so $E_0$ is measurable of $0$ measure. Notice that $g_h(t)\geq -2\sigma_2$ on $E_0$. Hence boundedness of $(B(r(t),t))'$ implies that $\mu \chi_{E_0}\geq (B(r(t),t))'$ in the sense of distribution. 
Next passing $h\to 0$ in \eqref{5.8} shows that
\[
    \mu+\sigma k(t)-(B(r(t),t))'
\]
is a non-negative measure on $[t_0,T]\backslash E_0$. Overall these yield that
\[\mu\geq -\sigma k(t)+(B(r(t),t))' \quad\text{ in $[t_0,T]$, in the sense of distribution.}\]

Combining this with \eqref{5.16}, we obtain
\begin{equation}\lb{5.166}
k'(t)+\sigma k(t)\geq 0
\end{equation}
in the sense of distribution for $t\in [t_0,T]$. Thus we obtain
\[
k(t)\geq e^{-\sigma (t-t_0)}k(t_0)
\]
which implies that $k(t)=-D_x^- u(r(t),t)>0$ for $t\in [t_0,T]$. Here the Gr\"{o}nwall type inequality for distributions can be justified as follows. Since $k$ is bounded and $k(t_0)>0$, \eqref{5.166} implies that $k(t)$ is positive for a short positive time after $t_0$. Then we can approximate $k$ by $k_\tau:=k*\phi_\tau$ where $\{\phi_\tau\}_{\tau>0}$ are smooth non-negative mollifiers satisfying $\phi_\tau \to \delta_0 \text{ as }\tau\to 0$. Then for all $\tau>0$ small enough, $k_\tau(t_0+\tau)\geq c>0$ with $c$ independent of $\tau$, and so the conclusion follows since $k_\tau$ satisfies \eqref{5.166} in the classical sense.

\medskip

Now we proceed to prove the claim.

\smallskip

\textbf{Proof of the claim}. 
To prove that $\underline{u}$ is a subsolution to \eqref{1.2}, by Lemma \ref{L.2.4}, it suffices to prove $\calL(\underline{u})\geq 0$
in the positive set of $\underline{u}$, where the operator $\calL$ is given in \eqref{calL}.

For abbreviation of notations, denote $\alpha=\alpha(t)$,
\begin{equation}
    \label{5.11}
    y=y(x,t):=x+\alpha_0-Bt+(BB_{x}+B_t)\frac{t^2}{2},
\end{equation}
and then $\underline{u}=\lambda(\alpha^2-y^2)_+$.
We have
\begin{equation}
    \label{5.13}
\begin{aligned}
    &y_x=1-B_xt+(B_{xx}B+B_{x}^2+B_{xt})\frac{t^2}{2},\\
    &y_{xx}=-B_{xx}t+(B_{xxx}B+3B_{xx}B_x+B_{xxt})\frac{t^2}{2},\\
    &y_t=-B+BB_x t+(B_tB_x+BB_{xt}+B_{tt})\frac{t^2}{2}.
\end{aligned}
\end{equation}

Plugging $\underline{u}$ into the operator $\calL$, we find in the positive set of $\underline{u}$ (i.e. $|y|< \alpha$) that,
\begin{align*}
\calL(\underline{u})
&=
    \lambda'(\alpha^2-{y}^2)+2\lambda (\alpha\alpha'-{y}y_t)+2(m-1)\lambda^2(y_{x}^2+{y}\,y_{xx})(\alpha^2-{y}^2)\\
    &\quad -4\lambda^2{y}^2y_{x}^2-2\lambda {y}\,y_x B+(m-1)\lambda(\alpha^2-y^2)B_x\\
&=\l\lambda'+2(m-1)\lambda^2(y_{x}^2+y\,y_{xx})+(m-1)\lambda B_x\r(\alpha^2-y^2)\\
&\quad +2\lambda\alpha \alpha' -2\lambda y\,y_t-4\lambda^2y^2y_{x}^2-2\lambda y\,y_xB.
\end{align*}
To have $\calL(\underline{u})\leq 0$, we only need to verify the following two inequalities
\begin{equation}\label{5.2}
 J_1:=   \frac{\lambda'}{2\lambda}+2(m-1)\lambda(y_{x}^2+y\,y_{xx})+(m-1) B_x\leq 0,
\end{equation}
and
\begin{equation}\label{5.3}
J_2:=\frac{\lambda'}{2}(\alpha^2-y^2)+2\lambda\alpha \alpha' 
-2\lambda y\,y_t-4\lambda^2 y^2\,y_{x}^2-2\lambda y\,y_x B\leq 0.
\end{equation}

By \eqref{5.11} and the regularity of $B$, there is $\tau_1=\tau_1(\sigma_2)$ such that for all $t\in (0,\tau_1)$,
\[|y_x-1|+|y_{xx}|<1.\]
Next recall the definitions of $\lambda(t)$ and $\alpha(t)$, and then we get 
\[\lambda(t)\leq \frac{\sigma_3}{2},\quad \alpha'(t)\leq k_0\leq \sigma_0\quad\text{ and }\quad\frac{\lambda'(t)}{2\lambda(t)}=-L.\]
Hence for $t\leq \frac{1}{\sigma_3}$,
\[\alpha(t)\leq \alpha_0+ k_0 t =\frac{2k_0}{\sigma_3}\leq \frac{2\sigma_0}{\sigma_3}.\]
Also in the support of $\underline{u}$, we have
$ |y(x,t)|<\alpha(t)\leq \frac{2\sigma_0}{\sigma_3}$ for $t\leq \frac{1}{\sigma_3}$.
Plugging these estimates, as well as \eqref{5.13}, into the left-hand side of \eqref{5.2} yields for $t\leq \frac{1}{\sigma_3}$,
\begin{align*}
J_1&= -L+(m-1)\sigma_3(4+\frac{2\sigma_0}{\sigma_3})+(m-1)\sigma_2\\
&\leq -L+(m-1)(5\sigma_3+2\sigma_0)\leq 0.
\end{align*}

Next we prove  \eqref{5.3}. It follows from \eqref{5.1}, for some universal $c\in(0,1)$ and all $0<t\leq \frac{c}{\sigma_3}$,
\[\left|-B_xt+(B_{xx}B+B_x^2+B_{xt})\frac{t^2}{2}\right|\leq \sigma_3t.\]
Pick $\tau_2:=\min\{\frac{c}{\sigma_3},\, \frac{1}{2L},\tau_1\}$, and we have for $t\leq \tau_2$
\begin{align*}
    2\alpha \alpha' \lambda-4\lambda^2{y}^2\,y_{x}^2& =4\lambda^2(\alpha\alpha_0(1-{L}t)-{y}^2(1-B_xt+(B_{xx}B+B_{x}^2+B_{xt})\frac{t^2}{2})^2)\\
    & \leq 4\lambda^2(\alpha^2(1-{{L}}t)-{y}^2(1-\sigma_3t)^2)\\
    &\leq 4\lambda^2(\alpha^2(1-{2\sigma_3}t)-2\sigma_3\alpha^2t-{y}^2(1-{2\sigma_3}t))\\
    &= 4\lambda^2(\alpha^2-{y}^2)(1-{2\sigma_3}t)-8\sigma_3\lambda^2\alpha^2t
\end{align*}
In the first inequality we used $\alpha\geq \alpha_0$, while in the second inequality we used $L\geq 4\sigma_3$.

By \eqref{5.1} and \eqref{5.13}, we have for $|y|\leq \alpha$,
\begin{align*}
    -2\lambda y\,y_t-2\lambda y\,y_x B&=-2\lambda y(B_{xx}B^2+B_{x}^2B+B_{t}B_{x}+2BB_{xt}+B_{tt})\frac{t^2}{2}\\
    &\leq {\sigma_2}\lambda |y| t^2\leq {\sigma_3}\lambda \alpha t^2.
\end{align*}
It follows that
\begin{equation}
    \label{5.12}
J_2    \leq (\alpha^2-y^2)(\frac{\lambda'}{2}+4\lambda^2(1-2\sigma_3t))-8\sigma_3\lambda^2\alpha^2t+\sigma_3\lambda\alpha t^2.
\end{equation}
Note that $\lambda\leq \frac{\sigma_3}{2}$ and $\frac{\lambda'}{\lambda}=-2L\leq -8\sigma_3$. Hence
\[(\alpha^2-y^2)(\frac{\lambda'}{2}+4\lambda^2(1-2\sigma_3t))\leq (\alpha^2-y^2)\lambda(-L+4\lambda)\leq 0.\]
Moreover when $t\leq \frac{1}{2L}$, by definition, $\lambda(t)\geq \frac{\sigma_3}{2e}$. Also since $\alpha(t)\geq \alpha_0=\frac{k_0}{\sigma_3}$, we get 
\[\lambda(t)\alpha(t)\geq  \frac{k_0}{2e}\quad\text{ for }t\leq \tau_2.\]
So
\[-8\sigma_3\lambda^2\alpha^2t+\sigma_3\lambda\alpha t^2=\sigma_3\lambda\alpha t(-\frac{4k_0}{e}+t)\leq 0\]
holds for all
$t\leq \min\{k_0,\tau_2\}$. Combining these with \eqref{5.12} implies that $J_2\leq 0$.

Finally we proved
\[\calL(\underline{u})\geq 0 \quad \text{ for all }
0<t\leq \min\{k_0,\tau_2\},\]
and here, $\tau_2$ only depends on $t_0,T$ and universal constants. 

\end{proof}

The following corollary uses the condition \eqref{1.6} which is weaker than the initial non-degeneracy assumption (it is weaker because the non-degeneracy \eqref{nondeg} at time $0$ corresponds to $\gamma\leq 1$ in \eqref{1.6}).

\begin{corollary}\label{C.5.2}
Let $\varrho$ be the solution to \eqref{main} with bounded, non-negative, compactly supported initial data $\varrho_0$, and $V,W\in C_{x}^3C_t^2(\mathbb{R}\times [0,\infty))$. Let $u$ be its pressure variable and suppose \eqref{1.6} holds. Then $-D_x^- u(r(t),t)>0$ for all $t>0$.
\end{corollary}
\begin{proof}
Fix any $t_0>0$.
Recall \eqref{5.15} by Proposition \ref{P.1.2}. We get $r(t_0)>X(r(\frac{t_0}{2}),\frac{t_0}{2};\frac{t_0}{2})$ and so there exists $t_1\in (\frac{t_0}2,t_0)$ such that
\[D_t^+ r(t_1)>B(r(t_1),t_1)\]
and thus $-D_x^- u(r(t_1),t_1)>0$. By Theorem \ref{T.5.1}, for all $t>t_1$, we have $-D_x^- u(r(t),t)>0$.
\end{proof}

\medskip

\section{Higher Regularity}\label{s.5}

With the knowledge of the fundamental estimate and non-degeneracy, the $C^{1,\alpha}$ regularity of the free boundary follows from Theorem 6.1 \cite{kim2018fb}.

\begin{theorem}\label{T.6.1}
Assume the conditions of Corollary \ref{C.5.2}. Then $r(t)$ is a $C^{1,\alpha}$ function for all $t>0$.
\end{theorem}
\begin{proof}
It follows from Corollary \ref{C.5.2} that the free boundary is non-degenerate for all $t>0$. Since $u_{xx}$ is bounded from below and $u$ is $C^2$ in $\{u>0\}$, 
\begin{equation}
    \label{B.2}
    -u_x(x,t)=-D_x^-u(r(t),t)+\int_{x}^{r(t)}u_{xx}(y,t)dy\geq -D_x^-u(r(t),t)-C(r(t)-x)>0,
\end{equation}
if $r(t)-x$ is sufficiently small.
Hence $u$ is locally uniformly monotone decreasing in the positive $x$ direction near the free boundary $(r(t),t)$. Next, as before, we treat $B(x,t)=-(V+W*\varrho)(x,t)$ as a function of $x,t$, which satisfies \eqref{5.1}. Thus all the conditions of Theorem 6.1  \cite{kim2018fb} are satisfied and the conclusion follows.
\end{proof}

Thanks to Theorem \ref{T.6.1}, we are able to write $r'(t)$ instead of $D_t^+ r(t)$, and for simplicity (with a slight abuse of notation) we will also write
\beq\lb{5.20}
u_x(r(t),t):=D_x^- u(r(t),t).
\eeq

\begin{corollary}\label{C.6.1}
Under the assumption of Theorem \ref{T.6.1}, the functions $u_x,u_t,u\, u_{xx}$ are continuous in $\bigcup_{t>0}\left(\overline{\{u(\cdot,t)>0\}}\times\{t\}\right)$. And 
\[u_t=u_{x}^2-u_xB-(m-1)uB_x,\,\quad u\, u_{xx}=0\quad\text{ on the free boundary.}\]
\end{corollary}

The proof is similar to the one of Theorem 2.3 in \cite{caffarelli1979regularity}. We will sketch it in the appendix. 

\medskip

Now we proceed to show high regularities of both the solution and the free boundary. The proof follows the line of the argument in \cite{LV12} where the $(PME)$ is studied. Let us remark that there is an alternative approach to high regularities which is given in \cite{LV30}. Both of the approaches are based on the non-degeneracy property, and the parallel statements in Theorem \ref{T.6.1} and Corollary \ref{C.6.1}. 

\medskip

As discussed in the introduction, firstly we show that $u_{xx}$ is bounded from above near the free boundary. For $t_0>0$, let $x_0=r(t_0)$ be the right-hand side free boundary point of $u$. Recall that $X(t)=X(x_0,t_0;t)$ is the streamline starting at $(x_0,t_0)$. We get $v(x,t):=u(x+X(t),t_0+t)$ is a solution to
\begin{equation}
    \label{6.8}
    v_t=(m-1)v v_{xx}+|v_x|^2-v_x \tilde{B}-v \tilde{B}_x,
\end{equation} 
where 
\[\tilde{B}(x,t):=B(x+X(t),t_0+t)-B(X(t),t_0+t).\]
We have $\tilde{B}(0,0)=0$ and from Lemma \ref{L.3.2}, for $t\in [0,T]$ and any $k\geq 0$,
\begin{equation}
    \label{6.10}
    \|\tilde{B}(\cdot,t)\|_{C_x^k(\bbR)}+
\|\tilde{B}_{t}(\cdot,t)\|_{C_x^k(\bbR)}+\|\tilde{B}_{tt}(\cdot,t)\|_{C_x^k(\bbR)}\leq C(k,t_0,T).
\end{equation} 

Write $\zeta(t)$ as the right-hand side free boundary of $v(\cdot,t)$, and then Lemma \ref{L.4.2} implies that
\begin{equation}\label{6.3}
    \text{$\zeta(t)$ is Lipchitz continuous with $\zeta(0)=0$.}
\end{equation}
Also non-degeneracy of $u$ translates to $-v_x(0,0)>0$.


\begin{lemma}\label{L.6.2}
In the above setting, there exist $C,\eta>0$ such that $v_{xx}\leq C$ in $R_{\eta}$, where
\[R_{\eta}:=\{(x,t)\in\bbR^2\,|\, \zeta(t)-\eta< y<\zeta(t),\, |t|< \eta\}.\]
\end{lemma}
\begin{proof}

In the set $\{v>0\}$, $p:=v_{xx}$ satisfies
\begin{align*}
    \calL_2(p)&:= p_t-(m-1)v p_{xx}-2m v_x p-(m+1)p^2+ p_x{\tilde{B}}+(m+1) p{\tilde{B}_x}\\
    &\quad +(2m-1) v_x\tilde{B}_{xx}-(m-1)v \tilde{B}_{xxx}=0.
\end{align*}

Denote $k_0:=-v_x(0,0)=-u_x(x_0,t_0)>0$ and then set
\begin{equation}\label{eps}
    \eps:=\min\left\{\frac{1}{20},\,\frac{k_0}{4(4m+1)}\right\}.
\end{equation}  
By Corollary \ref{C.6.1}, $v_x$ is continuous up to the free boundary. Also we have $\tilde{B}(0,0)=0$, and $\tilde{B}$ and $\zeta$ are Lipschitz continuous. So there exists $\sigma_0=\sigma_0(k_0,t_0,v)$ such that for all $\eta\in (0,\sigma_0)$, we have
\begin{equation}\label{6.4}
    |-v_x-k_0|\leq \eps,\quad |\tilde{B}(\zeta(\cdot),\cdot)|\leq \eps\quad\text{ in } \overline{R_{2\eta}}.
\end{equation}

Applying Lemma \ref{L.4.2} to $v$ shows that $\zeta'(t)=-v_x(\zeta(t),t)+\tilde{B}(\zeta(t),t)$. By Corollary \ref{C.6.1} again, we find for $t\in [-2\eta,2\eta]$,
\begin{equation}
    \label{6.7}
    \begin{aligned}
     |\zeta'(t)-k_0|&=|-v_x(\zeta(t),t)+\tilde{B}(\zeta(t),t)-k_0|\\
     &\leq |-v_x(\zeta(t),t)-k_0|+|\tilde{B}(\zeta(t),t)|\leq 2\eps.
    \end{aligned}
\end{equation}
Thus we get
\[(k_0-2\eps)(t+2\eta)\leq \zeta(t)-\zeta(-2\eta)\leq (k_0+2\eps)(t+2\eta).\] 
Set
\[\zeta_*(t):=\zeta(-2\eta)+(k_0+3\eps)(t+2\eta),\]
and then it follows that for $t\in [-2\eta,2\eta]$,
\begin{equation}
    \label{6.14}
    \begin{aligned}
     \zeta_*(t)-\zeta(t)&\geq (k_0+3\eps)(t+2\eta)-(k_0+2\eps)(t+2\eta)\\
     &\geq \eps(t+2\eta),
    \end{aligned}
\end{equation}
and
\begin{equation}\label{6.13}
\begin{aligned}
  \zeta_*(t)-\zeta(t)
&\leq (k_0+3\eps)(t+2\eta)-(k_0-2\eps)(t+2\eta)\\
&\leq 5\eta\, \eps(t+2\eta)\leq 20\eta\eps.
\end{aligned}  
\end{equation}


Now we construct a barrier for $p$ that is of the form
\[\phi(x,t):=\frac{\alpha}{\zeta(t)-x}+\frac{\beta}{\zeta_*(t)-x}\quad (\text{ with }\alpha,\beta>0)\]
where
$\beta:=\frac{k_0}{8(m+1)}$
and $\alpha$ is a constant in $ (0,\beta)$.

For abbreviation of notations, we write $\zeta,\zeta_*,\phi$ as $\zeta(t),\zeta_*(t),\phi(x,t)$ below.
We obtain in $R_{2\eta}$,
\begin{equation*}
    \begin{aligned}
    \calL_2(\phi)&\geq \frac{\alpha}{(\zeta-x)^2}\l -\zeta'-2(m-1)\frac{v}{\zeta-x}-2m v_x-2(m+1)\alpha\r\\
    &\quad +\frac{\beta}{(\zeta_*-x)^2}\l -\zeta_*'-2(m-1)\frac{v}{\zeta_*-x}-2mv_x-2(m+1)\beta\r\\
    &\quad +\frac{\alpha}{(\zeta-x)^2}\l  \tilde{B}+(m+1)(\zeta-x)\tilde{B}_x\r+\frac{\beta}{(\zeta_*-x)^2}\l \tilde{B}+(m+1)(\zeta_*-x)\tilde{B}_x\r\\
    &\quad +(2m-1)v_x\tilde{B}_{xx}-(m-1)v\tilde{B}_{xxx}.
\end{aligned}
\end{equation*}
In view of \eqref{6.10} and \eqref{6.3}, in $R_{2\eta}$
\[|\tilde{B}|\leq C\|B_x\||x|\leq C(|\zeta(t)|+2\eta)\leq C\eta \quad\text{ and }\quad\|\tilde{B}\|_{C_x^\infty}\leq C.\]
We then get
\begin{equation}
    \label{6.5}
    \begin{aligned}
    \calL_2(\phi)
    &\geq \frac{\alpha}{(\zeta-x)^2}\l -\zeta'-2(m-1)\frac{v}{\zeta-x}-2m v_x-2(m+1)\alpha-C\eta-C(\zeta-x)\r\\
    &\quad +\frac{\beta}{(\zeta_*-x)^2}\left(- \zeta_*'-2(m-1)\frac{v}{\zeta_*-x}-2mv_x-2(m+1)\beta-C\eta-C(\zeta_*-x)\right)\\
    &\quad -C(B,\|u\|_{C^1_x})\qquad \text{  in $R_{2\eta}$}.
\end{aligned}
\end{equation}

It follows \eqref{6.4} that for $(x,t)\in R_{2\eta} $, 
\[|v(x,t)|=|v(x,t)-v(\zeta(t),t)|\leq (k_0+\eps)(\zeta(t)-x),\]
which implies that
\begin{equation}
    \label{6.6}
    \left|\frac{v}{\zeta-x}-k_0\right|\leq \eps.
\end{equation}
Since $\zeta_*\geq \zeta$, we also have for such $(x,t)$,
\begin{equation}
\label{6.12}
    \frac{v}{\zeta_*-x}\leq k_0+ \eps.
\end{equation}
Next using \eqref{6.13} and $\eps<\frac{1}{20}$, we obtain
\begin{equation}
    \label{6.9}
\begin{aligned}
 \zeta_*-x\leq \zeta_*-\zeta+2\eta\leq 4\eta.
\end{aligned}   
\end{equation}

Let us apply \eqref{6.4}\eqref{6.7}\eqref{6.6}-\eqref{6.9} in \eqref{6.5} to get in $R_{2\eta}$,
\begin{align*}
    \calL_2(\phi)&\geq \frac{\alpha}{(\zeta-x)^2}\l -k_0-2\eps-2(m-1)(k_0+\eps)+2m(k_0-\eps)-2(m+1)\alpha-C\eta\r\\
    &\quad +\frac{\beta}{(\zeta_*-x)^2}\left( -k_0-3\eps-2(m-1)(k_0+\eps)+2m(k_0-\eps)-2(m+1)\beta-C\eta\right)-C\\
    &\geq \frac{\alpha}{4\eta^2}\l k_0-4m\eps-2(m+1)\alpha-C\eta\r\\
    &\quad +\frac{\beta}{4\eta^2}\left( k_0-(4m+1)\eps-2(m+1)\beta-C\eta-\frac{C\eta^2}{\beta}\right).
\end{align*}
Recall \eqref{eps} and \[\alpha<\beta=\frac{k_0}{8(m+1)}.\]
Then it is not hard to see that there exists $\sigma_1=\sigma_1(k_0,t_0,v)\leq \sigma_0$ (independent of $\alpha$) such that for all $\eta<\sigma_1$, we have $\calL_2(\phi)\geq 0$.

\medskip

Next we show $\phi\geq p=v_{xx}$ on the parabolic boundary of $R_{2\eta}$. By Corollary \ref{C.6.1}, $v\,v_{xx}\to 0$ as $(x,t)$ approaches the free boundary. And by \eqref{6.4}, $v(x,t)\geq (k_0-\eps)(\zeta-x)$ in $R_{2\eta}$. Thus we can fix $\eta\in (0,\sigma_1)$ to be small enough depending only on $k_0,t_0 $ and $v$ such that $v\,v_{xx}\leq \frac{\beta(k_0-\eps)}{1+10\eps}$ in $R_{2\eta}$, which yields
\begin{equation}\label{6.11}
    v_{xx}\leq \frac{\eps'}{(\zeta-x)}\quad\text{ in }R_{2\eta},\quad \text{ where }
\eps':=\frac{\beta }{10\eps+1}.
\end{equation}

Using \eqref{6.13} and \eqref{6.11}, we deduce for $t\in (-2\eta,2\eta)$,
\begin{align*}
    \phi(\zeta(t)-2\eta,t)&\geq\frac{\beta}{\zeta_*-\zeta+2\eta}\geq\frac{\beta}{20\eta\eps+2\eta}\\
    &\geq \frac{\eps'}{2\eta}\geq v_{xx}(\zeta(t)-2\eta,t).
\end{align*}
For $t=-2\eta$ and $x\in (\zeta(-2\eta)-2\eta,\zeta(-2\eta))$, due to \eqref{6.11} again,
\[\phi(x,-2\eta)\geq \frac{\beta}{\zeta(-2\eta)-x}\geq\frac{\beta}{2\eta}\geq v_{xx}(x,-2\eta).\]

Finally consider the right-hand side lateral boundary of $R_{2\eta}$. Due to Corollary \ref{C.6.1}, there is a neighbourhood depending on $\alpha$ (denoted as $N_\alpha$) of $\{(\zeta(t),t),\,\,|t|< 2\eta\}$ such that 
\[v\,v_{xx}\leq \alpha(k_0-\eps)\quad \text{ in }N_\alpha\cap R_{2\eta}.\]
It follows from \eqref{6.6} that $\frac{v}{\zeta-x}\geq k_0-\eps$, which implies
\[\phi\geq \frac{\alpha}{\zeta-x}\geq \frac{\alpha (k_0-\eps)}{v}\geq v_{xx}(x,t) \quad \text{ in }N_\alpha\cap R_{2\eta}.\]
Therefore by comparing $\phi$ and $v_{xx}$ in $ R_{2\eta}\backslash N_\alpha$, we get $\phi \geq v_{xx}$ in $R_{2\eta}\backslash N_\alpha$. From the above we proved $\phi \geq v_{xx}$ in $R_{2\eta}$.

Since $\eta$ is independent of $\alpha$, after passing $\alpha\to 0$, the order of $\phi,v_{xx}$ shows that
\[v_{xx}(x,t)\leq \frac{\beta}{\zeta_*-x}\quad \text{ in } R_{2\eta}.\]
By \eqref{6.14}, for $(x,t)\in R_{\eta}$, we have
$\zeta_*(t)-x\geq \eps\eta.$
We conclude with $v_{xx}(x,t)\leq \frac{\beta}{\eps\eta}$ in $R_{\eta}$.

\end{proof}
Lemma \ref{L.6.2} implies that $u_{xx}$ is bounded from above near $(r(t_0),t_0)$. Combining this with the fundamental estimate, we obtain that $|u_{xx}|$ is locally uniformly bounded near the free boundary if we have non-degeneracy.

\medskip


Now we estimate the higher derivatives of $u$ near the free boundary. As before, we consider $v$ instead of $u$. For $j\geq 1$, write $v^{(j)}:=\partial_x^j v$ and $\tilde{B}^{(j)}:=\partial_x^j\tilde{B}$.
Notice that for $j\geq 3$, $v^{(j)}$
satisfies the linear equation
\begin{align*}
    \calL_{j} (v^{(j)})&:= v^{(j)}_t-(m-1)v\, v^{(j)}_{xx}-(2+j(m-1)) v_x v^{(j)}_x-v^{(j)}_x\tilde{B}\\
    &\qquad -(j+1)v^{(j)}\tilde{B}^{(1)}-c^1\,v_{xx}\,v^{(j)}+\sum_{p=3}^{\lfloor j/2\rfloor+1}c^2_p\,v^{(p)}v^{(j+2-p)}+\sum_{p=0}^{j-1}c^3_p\,v^{(p)}\tilde{B}^{(j+1-p)}=0,
\end{align*}
where the constant $c^1$ only depends on $m,j$, and the constants $c_p^2,c_p^3$ only depend on $p,j,m$.

We have the following lemma:

\begin{lemma}\label{T.6.2}
Suppose $V,W\in C_{x,t}^\infty$ and \eqref{1.6} holds. For any $t_0>0$, let $x_0=r(t_0)$ and $v=u(x+X(x_0,t_0;t),t_0+t)$. For each integer $j\geq 2$, there exist positive constants $C_j,\eta_j$ depending on $m,d,j,t_0,V,W$ and $u$ such that $|v^{(j)}|\leq C_j$ in $R_{\eta_j}$, where $R_{\eta_j}$ is given in Lemma \ref{L.6.2}.
\end{lemma}

\begin{proof}

Recall that we write the right-hand side free boundary of $v$ as $\zeta$ and so $\zeta(0)=0$.

The proof proceeds by induction. By Lemma \ref{L.6.2} and the fundamental estimate $|v^{(2)}|<\infty$ in $R_{\eta}$ for some $\eta>0$. 
Suppose that for some $k\geq 2$, $|v^{(j)}|\leq C_k$ for all $j=2,3,...,k$ in $R_{\eta_k}$ for some $\eta_k>0$, and the goal is to show boundedness of $v^{(k+1)}$ in $R_{\eta_{k+1}}$ for some $\eta_{k+1}>0$.

Notice that the operator $\calL_{k+1}$ is of the form:
\begin{align*}
    \calL_{k+1}(\phi)&=\phi_t-(m-1)v\, \phi_{xx}+f_1 \phi_x+f_2\phi+f_3,
\end{align*}
where, by induction hypothesis, $f_1,f_2,f_3$ are bounded functions. This is of the same form for the cases when $B\equiv 0$. Therefore following the proof of Proposition 3.1, \cite{LV12}, there exist $\eta_{k+1}>0$ and $C_{k+1}>0$ such that
\[|v^{(k+1)}|\leq C_{k+1}\quad\text{ in }R_{\eta_{k+1}}.\]
And we can conclude.

\medskip

Let me briefly sketch the key idea used in Proposition 3.1 \cite{LV12} to prove the inductive step below.
Consider any subset $\overline{R}$ of $R_{\eta_{k}}$ such that $(x,t)\in \overline{R}$ implies $\zeta(t)-x\geq\lambda>0$ for some $\lambda>0$. By the non-degeneracy property, $v\geq c\lambda$ in $\overline{R}$ and therefore the operator $\calL_{k+1}$ is uniformly parabolic with elliptic constant $\geq c\lambda$ in $\overline{R}$. It then follows from the regularity estimate for parabolic type equation (Theorem 5.3.1 \cite{LSU}) that
$ |v^{(k+1)}|\leq \frac{C}{\lambda}$ in $\overline{R}$. This implies that 
\[|v^{(k+1)}|\leq \frac{C}{\zeta(t)-x}\quad \text{ in } R_{\eta_{k}}.\]

To remove the denominator $\frac{1}{\zeta-x}$, we apply the barrier transformation lemma (Lemmas 3.2-3.4 of \cite{LV12}) and the estimate can be improved to
\[|v^{(k+1)}|\leq C_{k+1}\quad \text{ in }R_{\eta_{k+1}}\text{ for some smaller }\eta_{k+1}>0.\]
\end{proof}

\textbf{Proof of Theorem \ref{T.1.4}.}

\medskip

 Due to Lemma \ref{L.3.2}, $B\in C_x^\infty C_t^{1,1}(\bbR\times [t_0-\eta,t_0+\eta])$. The classical parabolic regularity result yields that $u$ is $C^\infty_xC^2_t$ in $\Omega=\{(x,t)\,|\,u>0,t>0\}$. Below we need to obtain a uniform bound up to the boundary.
For any $k\geq 0$, we apply $(\frac{\partial}{\partial t})(\frac{\partial }{\partial x})^k$ to $\calL(u)=0$ where the operator $\calL$ is given in \eqref{calL}. We get that ${\phi_1}:=(\frac{\partial}{\partial t})(\frac{\partial }{\partial x})^ku$ satisfies the linear equation in $\Omega$ in the classical sense
\begin{equation}
    \label{5.10}
    (\phi_1)_t=(m-1)u(\phi_1)_{xx}+f^1_1(\phi_1)_x+f^1_2{\phi_1}+f^1_3
\end{equation}
where $f^1_1,f^1_2,f^1_3$ are linear combinations of
$\partial_x^pu,\, \partial_x^pB,\, \partial_x^qu\,\partial_x^p\partial_t B $ 
with $q+p\leq k+1$. By taking $\eta\in (0,\frac{t_0}{2})$ to be small, we can assume that $u$ is strictly positive in $N_{\eta}(t_0)$ when $(x,t)$ is away from the right-hand side free boundary. Since $u$ is smooth in the region where it is strictly positive, it then follows from Lemmas \ref{L.3.1}, \ref{T.6.2} that $u$ is Lipschitz continuous in time and spatially smooth uniformly in $N_{\eta}(t_0)$.

Due to Lemma \ref{L.3.2}, $B\in C_x^\infty C_t^{1,1}(\bbR\times [t_0-\eta,t_0+\eta])$.
Therefore in $N_{\eta}(t_0)$, the right-hand side of \eqref{5.10} is uniformly bounded  which implies that 
$\phi_1$ is Lipschitz continuous in time. We deduce that 
$u$ is uniformly $C_x^\infty{C}_t^{1,1}$ in $N_{\eta}(t_0)$.

\medskip

Next we can apply $(\frac{\partial}{\partial t})^2(\frac{\partial }{\partial t})^k$ to $\calL(u)=0$ to get that ${\phi_2}:=(\frac{\partial}{\partial t})^2(\frac{\partial }{\partial x})^ku$ satisfies 
\[(\phi_2)_t=(m-1)u(\phi_2)_{xx}+f^2_1(\phi_2)_x+f^2_2{\phi_2}+f^2_3\quad\text{ in $\Omega$}\]
where $f^2_1,f^2_2,f^2_3$ are bounded functions, due to the established regularities for $B,u$.
Then similarly we obtain that $u$ is $C_x^\infty{C}_t^{2,1}$ uniformly in $N_{\eta}(t_0)$. However we are not able to proceed further with this argument since $B$ is only known to be $C^{1,1}$ in time.

\medskip
Using this regularity of $u$, we claim that the free boundary is a $C^{2,1}$ function of $t$. Indeed, let us write $r(t)$ as the right-hand side free boundary of $u$ and then $u(r(t),t)=0$. 
Since it was proved that $r(t)$ is Lipschitz continuous, after differentiating the above equality by $t$, we get
\[u_x(r(t),t)r'(t)+u_{t}(r(t),t)=0.\]
By non-degeneracy, $u_x$ is strictly positive near the free boundary. Also since $u_x(r(t),t),u_{t}(r(t),t)$ are Lipschitz continuous, $r$ is uniformly $C^{1,1}$ in $(t_0-\eta,t_0+\eta)$. 

Then by differentiating the above equality one more time, we find
\[u_{x}(r(t),t)r''(t)+u_{xx}(r(t),t)|r'(t)|^2+u_{xt}(r(t),t)r'(t)+u_{tt}(r(t),t)=0.\]
It follows from the non-degeneracy property and the regularity established for $u$ that $r''(t)$ is Lipschitz continuous. Thus we proved that the free boundary $r(t)$ is a $C^{2,1}$ function locally uniformly for positive time.

\hfill $\Box$

\medskip

\textbf{Proof of Theorem \ref{T.1.5}.}

\medskip

In view of Lemma \ref{L.4.1} and the assumption that $\Omega_0$ is a finite interval, there are functions $l(t),r(t)$ such that $\Omega_t=(l(t),r(t))$ for all $t>0$.
The first part of the statement follows from Theorem \ref{T.1.4}.

To prove the second part, we need to improve the regularity of $B$ in time using the assumptions that $W\equiv 0$ or $\frac{1}{m-1}\in\bbN$.
For any $T>t_0>0$, let us always restrict $t$ to $(t_0,T)$ in this proof.
For induction, suppose that for some $p\geq 2$, $u$ is uniformly $C_x^\infty {C}_t^{p,1}$ in $\Omega\cap (\bbR\times\{t\in (t_0,T)\})$ and $l(\cdot),r(\cdot)\in C^{p,1}((t_0,T))$.
If $W\equiv 0$, by the assumption $B$ is smooth in space and time. In the case when $W\neq 0$ and $\frac{1}{m-1}\in\bbN$, by the assumption and inductive hypothesis, $u^{\frac{1}{m-1}}$ is bounded in $C_x^k{C}_t^{p,1}$ norms for any $k\geq 0$ in $\Omega\cap (\bbR\times \{t\in (t_0,T)\})$. Thus
\[B=-V-W*\varrho=-V-\l\frac{m-1}{m} \r^\frac{1}{m-1}\int_{l(t)}^{r(t)}W(x-y,t)u(y,t)^\frac{1}{m-1}\,dy\]
is also bounded uniformly in $C_x^k{C}_t^{p,1}(\bbR\times [t_0,T])$ norms.

Then applying the differential operator $(\frac{\partial}{\partial t})^{p+1}(\frac{\partial}{\partial x})^k$ to $\calL(u)=0$, we obtain that \[\phi_{p+1}:=(\frac{\partial}{\partial x})^k(\frac{\partial}{\partial t})^{p+1} u\] satisfies the following equality almost everywhere in $\Omega$,
\[(\phi_{p+1})_t=(m-1)u(\phi_{p+1})_{xx}+f^{p+1}_1(\phi_{p+1})_x+f^{p+1}_2\phi_{p+1}+f^{p+1}_3\]
where $f^{p+1}_1,f^{p+1}_2,f^{p+1}_3$ are uniformly Lipschitz continuous in $\Omega\cap(\bbR\times\{t\in (t_0,T)\})$. Therefore, by the equality and the inductive hypothesis, $(\phi_{p+1})_t$ is also Lipschitz in time. 

To conclude, we proved that when $W\equiv 0$ or $\frac{1}{m-1}\in\bbN$, $u$ is smooth in $\Omega\cap(\bbR\times\{t\in (t_0,T)\})$ uniformly up to the free boundary. 

\hfill $\Box$

\medskip

Finally, we show that permanent waiting time is possible with the appearance of either $V$ or $W$.

\medskip

\textbf{Proof of Theorem \ref{T.1.3}.}

\medskip


We will present two explicit examples of stationary solutions which indicate the possibility of permanent waiting time.

First let us consider the domain to be $\bbT$. Set 
\[\Phi:=-4\cos (2\pi x),\quad \varrho:=\sin (2\pi x)+1.\]
Then
\begin{align*}
    \Phi*\varrho(x)&=-4\int_0^1 \cos(2\pi(x-y))(\sin(2\pi y)+1)dy\\
    &=-2\int_0^1 \left[\sin(2\pi x)-\sin(2\pi x-4\pi y)+2 \cos(2\pi(x-y))\right] dy\\
    &=-2\sin(2\pi x).
\end{align*}
Therefore 
\[2\varrho+\Phi*\varrho=2.\]
Now if we pick $W:=\Phi_x$, then
\[(\varrho^2)_{xx}+(\varrho W*\varrho)_x=\left(\varrho\, (2\varrho+\Phi*\varrho)_x\right)_x=0.\]
Hence this pair of $\varrho, W$ satisfy 
\begin{equation*}
    (\varrho^2)_{xx}+(\varrho W*\varrho)_x=0 \quad \text{ in }\mathbb{T},
\end{equation*}
and clearly $\varrho(\frac{3}{4})=0$.

\medskip

Next we present the second example. Take any open subset of $\mathbb{R}$ and write it as a union of disjoint open intervals: $\bigcup_{i\in\bbN} I_i$. For each $i$, let $\Psi_i$ be a smooth function such that $\Psi_i<0$ in $I_i$ and $\Psi_i=0$ outside $I_i$.
Next set 
\[\varrho_i:=\frac{m-1}{m}(-\Psi_i)^{\frac{1}{m-1}},\quad V_i:=(\Psi_i)_x,\]
and then
\[(\varrho_i^m)_{xx}+(\varrho_i V_i)_x=\left(\varrho_i \left(\frac{m}{m-1}\varrho_i^{m-1}+\Psi_i\right)_x\right)_x=0.\]
Therefore $\varrho_*:=\Sigma_{i\in\bbN}\, \varrho_i$ and $V:=\Sigma_{i\in\bbN} V_i$ is a pair of functions satisfying 
\begin{equation*}
    \label{1.4}
    (\varrho_*^m)_{xx}+(\varrho_* V)_x=0 \quad \text{ in }\mathbb{R},
\end{equation*}
and $\{\varrho_*=0\}=(\bigcup_{i\in\bbN} I_i)^c$.

Let $\varrho_0$ be a function such that $\varrho_0\leq \varrho_*$ and $\{\varrho_0>0\}=\{\varrho_*>0\}$. Then by comparison principle, the solution $\varrho(x,t)$ to \eqref{1.1} with initial data $\varrho_0$ satisfies $\varrho(\cdot,t)\leq \varrho_*(\cdot)$ for all $t\geq 0$. Hence $\Gamma_t(\varrho)\subseteq \overline{\{\varrho_*>0\}}$. Notice that by the construction, $V=0$ at $\Gamma_0=\bigcup_{i\in\bbN}\,(\partial I_i)$. It follows from Lemma \ref{L.4.1} that the free boundary is non-contracting along streamlines. Hence $\Gamma_t=\Gamma_0$ for all $t>0$.

\medskip

\appendix

\section{Proof of Proposition \ref{P.1.2}}

Using that $V,W\in C^3_xC^0_t(\bbR\times [0,\infty))$, $\varrho\geq 0$, and $\|\varrho(\cdot,t)\|_{L^1(\bbR)}=\|\varrho_0\|_{L^1(\bbR)}$ from the equation, we get that $B=B(x,t)$ is a vector field uniformly bounded in $C^3_xC^0_t(\bbR\times [0,\infty))$.
Note that the proof of Theorem 1.2 \cite{kim2018fb} only used $C^3_xC^0_t$ regularity of the vector field, and so the theorem yields either of the following holds (using the notation $X(t):=X(x_0,t_0;t)$)
\begin{itemize}
    \item[\normalfont{(1)} ] $X(-s)\in\Gamma$ for all $s\in [0,t_0]$ with $\Gamma:=\bigcup_{t>0}(\Gamma_t\times\{t\})$;
    \medskip
    \item[\normalfont{(2)} ] there exist $C,\beta>1$ and $h>0$ such that for $s\in(0,h)$, $u(x,t_0-s)=0$ if $|x-X(-s)|\leq Cs^\beta$, and $u(x,t_0+s)>0$ if $|x-X(s)|\leq Cs^\beta$.
\end{itemize}
Since the right-hand side free boundary is $r(t)$, (1) clearly implies (i) after taking $s=t_0-t$. Also we see that if $r(t)$ is of type (1), then $r(s)$ is of type (1) for all $s\in[0,t]$. 
Thus there exists $t_1\in [-1,\infty]$, such that $r(t)$ is of type (1) for all $t\in [0, t_1]$, and it is of type (2) for all $t>t_1$. 

Now suppose (2) holds at $t_0>0$, we get for all $s\in (0,h]$ (for some small $h>0$ depending on $t_0$) such that $r(t_0-s)<X(-s)$ and $r(t_0+s)>X(s)$. We are going to show (ii). Since all $r(t)$ with $t\geq t_0$ are of type (2) with $h$ depending only on $t_0$ and universal constants by \cite{kim2018fb}, we use the property that 
\[
X(r(t_0+h),t_0+h;s')>X(X(h),t_0+h;s')= X(x_0,t_0;h+s') \quad\text{ for all $s'>0$}
\]
to conclude that $r(t_0+h+s')>X(r(t_0+h),t_0+h;s')> X(h+s')$ for all $s\in(0,h]$. Therefore, by iteration, we obtain
$r(t_0+s)>X(s)$ for all $s>0$ which yields the first part of (ii) with $s=t-t_0$. For the other part, assume for contradiction that there is $s\in (0,t_0]$ such that $r(t_0-s)\geq X(-s)$. It follows from (i) and the first part of (ii) (or, alternatively, Lemma \ref{L.4.1}) that for any $z\in (0,s)$,
\[
r(t_0-z)\geq X(X(-s),t_0-s;s-z)=X(-z),
\]
which cannot happen because $r(t_0)$ is of type (2). 
Overall we obtained (ii), which concludes the proof.



\section{Proof of Proposition \ref{T.2.3}}

By Theorem \ref{exitence boundedness}, $\varrho\in L^\infty(\bbR\times[0,\infty))$ and so does $u$.
Recall $B$ in \eqref{B} and since $V,W$ are smooth, $B$ is smooth in space. Theorem 1.1 in \cite{ zhang2017regularity} implies that the solution is H\"{o}lder continuous for $t>0$ and so $B=B(x,t)$ is also H\"{o}lder continuous for $t>0$. 


We proceed by considering a set of approximated solutions $u_k$ with $k\in\bbZ^+$.
Take smooth approximations $B_k$ of $B$ and smooth non-negative approximations $u_{0,k}$ of $u_0$. Let $u_k$ be the solution to \eqref{1.2} with vector field $B_k$ and initial data $u_{0,k}+\frac{1}{k}$. By comparison principle, $u_k$ is positive, and so \eqref{1.2} is locally uniformly parabolic for all finite time. Then by the standard parabolic theory, $u_k$ is smooth. Parallel to the proof of Lemma 9.5 \cite{book}, we can show that $u_k$ are positive and smooth, and $u_k\to u$ locally uniformly. Therefore to prove the proposition, it suffices to consider positive and smooth $u$.

Set $p:=u_{xx}$, and then by differentiating \eqref{1.2} twice, we get
\begin{align*}
    p_t&=(m-1)u  p_{xx}+2mu_xp_x+(m+1)p^2\\
    &-p_x{B}-(m+1)p{B_x}-(2m-1)u_x B_{xx}+(m-1)uB_{xxx}
\end{align*}
By Young's inequality, we have
\begin{align*}
    \left|(m+1)pB_{x}\right|&\leq {m}\,p^2+Cm \\
    \left|(2m-1)u_xB_{xx}\right|&\leq {m}\,u_x^2+Cm,\\
    \left|(m-1)u B_{xxx}\right|
    &\leq Cm.
\end{align*}
Thus we obtain 
\[p_t-(m-1)u  p_{xx}-2m\,u_x p_x-p^2+p_x{B}+m\,u_x^2+Cm\geq 0.\]
Viewing $u$ as a known function, we can write the above quasilinear parabolic operator of $p$ as $\calL_0(p)$, and thus $\calL_0(p)\geq 0$.

Take $w=-\frac{1}{t+\tau}+u-C_1$ for $\tau>0$ and $C_1\geq \|u\|_\infty$ to be determined later. Then
\[\calL_0(w)=\frac{1}{(t+\tau)^2}+u_t-(m-1)u u_{xx}-m\, u_x^2-\l-\frac{1}{t+\tau}+u-C_1\r^2+u_x{B}+Cm.\]
Now we use the equation \eqref{1.2} and the fact that $u, B_x$ are bounded to get for some $C>0$ that
\begin{align*}
    \calL_0(w)&\leq \frac{1}{(t+\tau)^2}-\left(m-1\right) u_x^2-\l-\frac{1}{t+\tau}+u-C_1\r^2+Cm\\
    &\leq \frac{1}{(t+\tau)^2}-\l\frac{1}{t+\tau}+C_1-u\r^2+Cm\\
    &\leq -(C_1-u)^2+Cm\leq 0,
\end{align*}
if $C_1$ is large enough depending only on $m,\|u\|_\infty,\|B\|_{C_x^3C_t^0}$. 
Therefore $\calL_0(w)\leq 0\leq \calL_0(p)$. And we know $p(\cdot,0)\leq w(\cdot,0)$ since $w(\cdot,0)\to-\infty$ as  $\tau\to 0$. By comparison, we have
\[ u_{xx}=p\geq w \geq -\frac{1}{t}-C_1.\]

\medskip

\section{Proof of Lemma \ref{L.2.2}}

By Theorem \ref{exitence boundedness}, the solution is uniformly bounded.
Denote 
\begin{equation}\label{B.0}
    M:=\|B\|_\infty+\|B_x\|_\infty+\|u\|_\infty.
\end{equation}
Suppose $u_0$ is supported in $(-\infty,R)$. 
Take
\[\alpha:=(m-1)M,\quad  \tau=\frac{1}{\alpha},\quad C_1=(e+1)M\tau+1.\]
Let us prove by induction that 
\begin{equation}\label{B.1}
    u(\cdot,t)\text{ is supported in }(-\infty, R+C_1n] \text{ if }t\in [0, n\tau].
\end{equation}

When $n=0$, \eqref{B.1} holds by the assumption. Suppose \eqref{B.1} holds with $n=k$ for some $k \in \mathbb{N}$. Because $u(\cdot,k\tau)$ is supported in $(-\infty,R+C_1k]$ and $u$ is bounded by $M$, then
\[\phi(x,t):= e^{\alpha t}M(R+C_1k+(e+1)M t+1-x)_+\]
satisfies that
\[\phi(\cdot,0)\geq u(\cdot,k\tau)\quad\text{ on }\mathbb{R}.\]

Using \eqref{B.0}, direct computation yields that in the positive set of $\phi$,
\begin{align*}
&    \phi_t-(m-1)\phi\phi_{xx}-|\phi_x|^2+\phi_xB(\cdot,\cdot+k\tau)+(m-1)\phi B_x(\cdot,\cdot+k\tau)\\
\geq\, &\alpha \phi+e^{\alpha t}(e+1)M^2-e^{2\alpha t}M^2-e^{\alpha t}M^2 -(m-1)M\phi\\
\geq\,& 0,
\end{align*}
if $t\in [0,\tau]=[0,\frac{1}{\alpha}]$. Since $\phi$ is Lipschitz continuous, Lemma \ref{L.2.4} implies that $\phi$ is a supersolution for $t\in [0,\tau]$. Thus it follows from the comparison principle that
\[u(\cdot,\cdot+k\tau)\leq \phi(\cdot,\cdot) \quad \text{ in }\mathbb{R}\times [0,\tau].\]

Since for $t\in [0,\tau]$, the right end-point of the support of $\phi$ is bounded from above by 
\[R+C_1k+(e+1)M \tau+1=R+C_1(k+1).\]
We obtain that $u(\cdot,t)$ is supported inside $(-\infty,R+C_1(k+1)]$ for $t\leq (k+1)\tau$. By induction, we established \eqref{B.1}.

Similarly we can get a lower bound on the left end-point of the support of $u$.
We conclude that there exists $C>0$ depending only on $u_0$ and $\|B\|_{C_x^1C^0_t}$ such that $u(\cdot,t)$ is supported in $(-C(1+t),C(1+t))$.

\medskip

\section{Proof of Corollary \ref{C.6.1}}

First we prove that $u_x$ is continuous in $\Omega$ up to a free boundary point $(r(t_0),t_0)$ with $t_0>0$. By Theorem \ref{T.6.1}, $r'(t)$ is continuous. Then using the notation \eqref{5.20} and Lemma \ref{L.4.2} yields $r'(t)=-u_x(r(t),t)+B(r(t),t)$. Therefore
$u_x(r(t),t)$ is continuous in $t$.
In view of \eqref{B.2}, we obtain
\[\limsup_{(x,t)\to (r(t_0),t_0)}u_x(x,t)\leq u_x(r(t_0),t_0).\]

For the other direction, we prove by contradiction. Denote $x_0:=r(t_0)$ and $k_0=-u_x(x_0,t_0)$. Suppose there is a sequence of $(x_n,t_n)\in\Omega$ and $(x_n,t_n)\to (x_0,t_0)$ as $n\to \infty$ such that for some $\delta>0$,
\[u_x(x_n,t_n)\geq -k_0+\delta.\]
Then by the fundamental estimate,  for any $x<x_n$,
\[u(x,t_n)\geq u(x_n,t_n)+(-k_0+\delta)(x-x_n)-C(x-x_n)^2.\]
After passing $n\to \infty$, we get
\[u(x,t_0)\geq (-k_0+\delta)(x-x_0)-C(x-x_0)^2\]
which is impossible since $-u_x(x,t_0)\to k_0$ as $x\to x_0$ (due to Lemma \ref{L.4.2}).
The continuity of $u_x$ at $(x_0,t_0)$ follows.

Now we only need to show the continuity of $u_t$ at the free boundary, because after that the continuity of $u\, u_{xx}$ follows due to \eqref{1.2}. 
By the equation,
\[
\liminf_{(x,t)\to(x_0,t_0)}u_t=\liminf_{(x,t)\to(x_0,t_0)}((m-1)u\,u_{xx}+|u_x|^2-u_xB-(m-1)uB_x) \geq k_0^2+k_0B(x_0,t_0),
\]
where in the inequality we used $u(x_0,t_0)=0$, $u_{xx}\geq -C$ and $\lim_{(x,t)\to (x_0,t_0)}u_x(x,t)=-k_0$. This proves one side of the desired continuity of $u_t$.

Based on non-degeneracy, one can argue as in Lemma 4.2 of \cite{caffarelli1979regularity} with the help of Schauder estimates that for some $\eta>0$ and $C>0$,
\[|u\, u_{tt}|\leq C\quad\text{ in }N_\eta(t_0)\]
where $N_\eta(t_0)$ is defined in \eqref{Neta}.
Now suppose for contradiction that there is a sequence $(x_n,t_n)\to(x_0,t_0)$ such that for some $\delta>0$,
\[u_t(x_n,t_n) \geq k_0^2+k_0B(x_0,t_0)+\delta.\]

Denote $\eps_n:=r(t_n)-x_n$ which converges to $0$ as $n\to \infty$.
Next for any $\theta\in (0,1)$,
\begin{align*}
    u(x_n,t_n+\theta\eps_n)&=u(x_n,t_n)+u_t(x_n,t_n)\theta\eps_n+\frac{1}{2}u_{tt}(x_n,\xi)(\theta\eps_n)^2
\end{align*}
where $\xi\in (t_n-\theta\eps_n,t_n+\theta\eps_n)$. 
Since $r$ is Lipschitz continuous, there exists $\theta_0\in (0,1)$ independent of $n$ such that for all $\theta\leq\theta_0$,
\[r(\zeta)-x_n\geq r(t_n)-x_n-C\theta\eps_n=\frac{\eps_n}{2}.\]
By non-degeneracy, there exists $c>0$ such that
\[u(x_n,\xi)\geq 2c(r(\xi)-x_n)\geq c(r(t_n)-x_n)=c\eps_n\]
and so $|u_{tt}(x_n,\xi)|\leq \frac{C}{c\eps_n}$.

Using the notation $k_n:=-u_x(r(t_n),t_n)$,
the fundamental estimate implies
\[u(x_n,t_n)    \geq k_n\eps_n-C\eps_n^2.\]
Then
\begin{equation}\label{C.0}
\begin{aligned}
    u(x_n,t_n+\theta\eps_n)
&    \geq u(x_n,t_n)+(k_0^2+k_0B(x_0,t_0)+\delta)\theta\eps_n-\frac{C(\theta\eps_n)^2}{2c\eps_n}\\
&    \geq k_n\eps_n-C\eps_n^2+(k_0^2+k_0B(x_0,t_0)+\delta)\theta\eps_n-C'\theta^2\eps_n\\
&\geq k_n\eps_n+(k_0^2+k_0B(x_0,t_0))\theta\eps_n+\frac{\delta\theta\eps_n}{2}
\end{aligned}
\end{equation}
if $\eps_n\leq \frac{\delta\theta}{4C}$ and $\theta\leq \frac{\delta}{4C'}$. Let us fix $\theta$ to be $\min\{\theta_0,\frac{\delta}{4C'}\}$.

By continuity of $u_x$ on the free boundary, there is $k_n'$ such that $k_n'\to k_0$ and
\[u(x_n,t_n+\theta\eps_n)\leq k_n'(r(t_n+\theta\eps_n)-x_n)=k_n'(r(t_n+\theta\eps_n)-r(t_n)+\eps_n).\]
Since $r'(t_0)=k_0+B(x_0,t_0)$ and $r'$ is continuous, then
\[r(t_n+\theta\eps_n)=r(t_n)+(k_0+B(x_0,t_0))\theta\eps_n+o(\eps_n).\]
We obtain
\begin{equation}\label{C.1}
    u(x_n,t_n+\theta\eps_n)\leq k_n'(k_0+B(x_0,t_0))\theta\eps_n+k_n'\eps_n+o(\eps_n).
\end{equation}

Combining \eqref{C.0} and \eqref{C.1} shows
\[(k_n'-k_n)(\eps_n+(k_0+B(x_0,t_0)\theta\eps_n))+o(\eps_n)\geq \frac{\delta\theta\eps_n}{2}\]
and then we get
\[(k_n'-k_n)(1+(k_0+B(x_0,t_0)\theta))+o(1)\geq \frac{\delta\theta}{2},\]
which is impossible after sending $n\to \infty$.
Therefore we proved 
\[\lim_{n\to \infty} u_t(x_n,t_n)= k_0^2+k_0B(x_0,t_0).\]

\bigskip


\begin{thebibliography}{10}



\bibitem{alt1983quasilinear}
H.~W. Alt and S. Luckhaus,
\newblock Quasilinear elliptic-parabolic differential equations.
\newblock {\em Math. Z.}, 183(3):311--341, 1983.

\bibitem{ambrosio2008gradient}
L. Ambrosio, N. Gigli and G. Savar{\'e},
\newblock {\em Gradient flows: in metric spaces and in the space of probability
  measures}.
\newblock Springer Science \& Business Media, 2008.

\bibitem{LV1}
S. Angenent,
\newblock Analyticity of the interface of the porous media equation after the
  waiting time.
\newblock {\em 	
Proc. Amer. Math. Soc.},
  102(2):329--336, 1988.

\bibitem{fundamentalest}
D.~G. Aronson and P. B{\'e}nilan,
\newblock R{\'e}gularit{\'e} des solutions de l'{\'e}quation des milieux poreux
  dans $\bbR^n$.
\newblock {\em C. R. Acad. Sci. Paris S{\'e}r. 1}, 288(2):103--105, 1979.

\bibitem{LV12}
D.~G. Aronson and J.~L. V{\'a}zquez.
\newblock Eventual $C^\infty$-regularity and concavity for flows in
  one-dimensional porous media.
\newblock {\em Arch. Ration. Mech. Anal.}, 99(4):329--348, 1987.

\bibitem{bedrossian2011local}
J. Bedrossian, N. Rodr{\'\i}guez and A.~L. Bertozzi.
\newblock Local and global well-posedness for aggregation equations and
  patlak--keller--segel models with degenerate diffusion.
\newblock {\em Nonlinearity}, 24(6):1683, 2011.






\bibitem{bertozzi2009existence}
A.~L. Bertozzi and D. Slepcev.
\newblock Existence and uniqueness of solutions to an aggregation equation with
  degenerate diffusion.
\newblock {\em Commun. Pure Appl. Anal.}, 9(6):1617, 2009.



\bibitem{TBL}
A.~ L.~Bertozzi, C. M.~Topaz and M.~A. Lewis,
\newblock A nonlocal continuum model for biological aggregation.
\newblock {\em Bull. Math. Biol.}, 68(7):1601--1623, 2006.

\bibitem{BGHP}
M.~Bertsch, M.~E.~Gurtin, D.~Hilhorst and L.~A.~Peletier,
\newblock On interacting populations that disperse to avoid crowding: The
  effect of a sedentary colony.
\newblock {\em J. Math. Biol.}, 19(1):1--12, 1984.

\bibitem{bertsch}
M. Bertsch and D. Hilhorst,
\newblock A density dependent diffusion equation in population dynamics:
  stabilization to equilibrium.
\newblock {\em SIAM J. Math. Anal.}, 17(4):863--883, 1986.

\bibitem{blanchet2009critical}
A. Blanchet, J.~A. Carrillo and P. Lauren{\c{c}}ot.
\newblock Critical mass for a patlak--keller--segel model with degenerate
  diffusion in higher dimensions.
\newblock {\em Calc. Var. Partial Differential Equations},
  35(2):133--168, 2009.

\bibitem{caffarelli1979regularity}
L.~A. Caffarelli and A. Friedman.
\newblock Regularity of the free boundary for the one-dimensional flow of gas
  in a porous medium.
\newblock {\em Amer. J. Math.}, 101(6):1193--1218, 1979.

\bibitem{CFregularity}
L.~A. Caffarelli and A. Friedman,
\newblock Regularity of the free boundary of a gas flow in an $n$-dimensional
  porous medium.
\newblock {\em Indiana Univ. Math. J.}, 29(3):361--391, 1980.


\bibitem{CVWlipschitz}
L.~A. Caffarelli, J.~L. V{\'a}zquez, and N.~I. Wolanski,
\newblock Lipschitz continuity of solutions and interfaces of the
  $n$-dimensional porous medium equation.
\newblock {\em Indiana Univ. Math. J.}, 36(2):373--401, 1987.

\bibitem{C1alpha}
L.~A. Caffarelli and N.~I. Wolanski,
\newblock $C^{1,\alpha}$ regularity of the free boundary for the $n$-dimensional
  porous media equation.
\newblock {\em Comm. Pure Appl. Math.}, 43(7):885--902,
  1990.
  
  
  \bibitem{carrillo2019aggregation}
J.~A. Carrillo, K. Craig and Y. Yao,
\newblock Aggregation-diffusion equations: dynamics, asymptotics, and singular
  limits.
\newblock {\em Active Particles, Volume 2}, pages 65--108. Springer, 2019.

\bibitem{carrillo2019}
J.~A. Carrillo and R.~S. Gvalani,
\newblock Phase transitions for nonlinear nonlocal aggregation-diffusion
  equations.
\newblock {\em Comm. Math. Phys.},  382(1):485-545, 2021.

\bibitem{carrillo2001entropy}
J.~A. Carrillo, A. J{\"u}ngel, P.~A. Markowich, G. Toscani
  and A. Unterreiter,
\newblock Entropy dissipation methods for degenerate parabolic problems and
  generalized sobolev inequalities.
\newblock {\em Monatsh. Math.}, 133(1):1--82, 2001.



\bibitem{chayes2013aggregation}
L. Chayes, I. Kim and Y. Yao,
\newblock An aggregation equation with degenerate diffusion: Qualitative
  property of solutions.
\newblock {\em SIAM J. Math. Anal.}, 45(5):2995--3018, 2013.



\bibitem{di1982continuity}
E.~DiBenedetto,
\newblock Continuity of weak solutions to certain singular parabolic equations.
\newblock {\em Ann. Mat. Pura Appl.}, 130(1):131--176, 1982.

\bibitem{dib83}
E.~DiBenedetto,
\newblock Continuity of weak solutions to a general porous medium equation.
\newblock {\em Indiana Univ. Math. J.}, 32(1):83--118, 1983.


\bibitem{LV30}
K.~H{\"o}llig and H.~O. Kreiss,
\newblock $C^\infty$-regularity for the porous medium equation.
\newblock {\em Math. Z.,}, 192(2):217--224, 1986.

\bibitem{hwang2019continuity}
S. Hwang and Y.~P. Zhang,
\newblock Continuity results for degenerate diffusion equations with 
  $ L^{p}_t L^{q}_{x}  $ drifts.
\newblock {\em Nonlinear Anal.}, 211:112413, 2021.


  
\bibitem{kienzler2018flatness}
C. Kienzler, H. Koch and J.~L. V{\'a}zquez,
\newblock Flatness implies smoothness for solutions of the porous medium
  equation.
\newblock {\em Calc. Var. Partial Differential Equations},
  57(1):18, 2018.

\bibitem{kim2018fb}
I.~C. Kim and Y.~P. Zhang,
\newblock Porous medium equation with a drift: Free boundary regularity.
\newblock {\em Arch. Ration. Mech. Anal.}, 1--52, 2021.

\bibitem{zhang2017regularity}
I.~C. Kim and Y.~P. Zhang,
\newblock Regularity properties of degenerate diffusion equations with drifts.
\newblock {\em SIAM J. Math. Anal.}, 50(4):4371--4406, 2018.

\bibitem{kimlei}
I.~C. Kim and H.~K. Lei,
\newblock Degenerate diffusion with a drift potential: A viscosity solutions
  approach.
\newblock {\em Discrete Contin. Dyn. Syst.}, 27(2):767--786, 2010.

\bibitem{Knerr}
B.~F. Knerr,
\newblock The porous medium equation in one dimension.
\newblock {\em Trans. Amer. Math. Soc.},
  234(2):381--415, 1977.

\bibitem{koch}
H. Koch,
\newblock {\em Non-Euclidean singular integrals and the porous medium
  equation}.
\newblock PhD thesis, Verlag nicht ermittelbar, 1998.

\bibitem{LSU}
O.~A. Lady{\v{z}}enskaja, V.~A. Solonnikov, and N.~N.
  Ural'ceva.
\newblock {\em Linear and quasi-linear equations of parabolic type},
\newblock American Mathematical Soc., 1988.

\bibitem{LV}
K.~A. Lee and J.~L.~V{\'a}zquez,
\newblock Geometrical properties of solutions of the porous medium equation for
  large times.
\newblock {\em Indiana Univ. Math. J.}, pages 52(4):991--1016, 2003.

\bibitem{numerics}
L. Monsaingeon,
\newblock Numerical investigation of the free boundary regularity for a
  degenerate advection-diffusion problem.
\newblock {\em Interfaces Free Bound.}, 19(3):371--391, 2017.

\bibitem{traveling}
L. Monsaingeon, A. Novikov and J.-M. Roquejoffre,
\newblock Traveling wave solutions of advection--diffusion equations with
  nonlinear diffusion.
\newblock {\em Ann. Inst. H. Poincar{\'e} Anal. Non Lin{\'e}aire}, 30(4):705--735, 2013.


\bibitem{book}
J.~L. V{\'a}zquez,
\newblock {\em The porous medium equation: mathematical theory}.
\newblock Oxford University Press, 2007.


\bibitem{zhang2018continuity}
Y.~P. Zhang,
\newblock On continuity equations in space-time domains.
\newblock {\em Discrete Contin. Dyn. Syst.}, 38(10):4837--4873,
  2018.

\bibitem{zhang2018class}
Y.~P. Zhang,
\newblock On a class of diffusion-aggregation equations.
\newblock {\em Discrete Contin. Dyn. Syst.}, 40(2):907, 2020.

\end{thebibliography}


\end{document}